\documentclass{article}
\usepackage{amssymb,amsmath,amsthm}
\usepackage{yhmath}
\usepackage{cite}
\newtheorem{theorem}{Theorem}[section]
\newtheorem{proposition}[theorem]{Proposition}
\theoremstyle{remark}
\newtheorem{remark}[theorem]{Remark}
\theoremstyle{definition}
\newtheorem{example}[theorem]{Example}

\numberwithin{equation}{section}
\allowdisplaybreaks[3]

\renewcommand{\L}{\mathcal{L}}
\newcommand{\I}{\mathcal{I}}

\newcommand{\K}{\mathcal{K}}
\newcommand{\X}{{\R^d}}
\newcommand{\R}{\mathbb{R}}
\newcommand{\N}{\mathbb{N}}
\newcommand{\B}{\mathcal{B}}
\newcommand{\n}{{(n)}}
\newcommand{\la}{\lambda}
\newcommand{\La}{\Lambda}
\newcommand{\Ga}{\Gamma}
\newcommand{\ga}{\gamma}

\newcommand{\eps}{\varepsilon}
\newcommand{\Bb}{\mathcal{B}_{\mathrm{b}}({\X})}
\newcommand{\Bbs}{B_{\mathrm{bs}}(\Ga_0)}
\newcommand{\Fc}{\mathcal{F}_\mathrm{cyl}(\Ga)}
\DeclareMathOperator*{\esssup}{ess\,sup\,}
\renewcommand{\ll}{\langle\!\langle}
\newcommand{\rr}{\rangle\!\rangle}

\begin{document}

\title{Binary jumps in continuum. II. Non-equilibrium process and a Vlasov-type scaling limit}

\author{Dmitri Finkelshtein%
\thanks{Institute of Mathematics, National Academy of Sciences of Ukraine, Kyiv,
Ukraine (\texttt{fdl@imath.kiev.ua}).}
\and Yuri Kondratiev%
\thanks{Fakult\"{a}t f\"{u}r Mathematik, Universit\"{a}t Bielefeld, 33615
Bielefeld, Germany (\texttt{kondrat@math.uni-bielefeld.de})} \and Oleksandr Kutoviy%
\thanks{%
Fakult\"{a}t f\"{u}r Mathematik, Universit\"{a}t Bielefeld, 33615
Bielefeld, Germany (\texttt{kutoviy@math.uni-bielefeld.de}).} \and
Eugene Lytvynov\thanks{Department of Mathematics, Swansea University, Singleton Park, Swansea SA2 8PP, U.K.
(\texttt{e.lytvynov@swansea.ac.uk})}}

\maketitle

\begin{abstract}
Let $\Gamma$ denote the space of all locally finite subsets
(configurations) in $\R^d$.  A stochastic dynamics of binary jumps
in continuum is a Markov process on $\Gamma$ in which pairs of
particles simultaneously hop over $\R^d$. We discuss a
non-equilibrium dynamics of binary jumps. We prove the existence of
an evolution of correlation functions on a finite time interval. We
also show  that a Vlasov-type mesoscopic scaling for such a dynamics
leads to a generalized Boltzmann non-linear equation for the
particle density.
\end{abstract}

\vspace{1.5mm}


\pagestyle{myheadings} \thispagestyle{plain}
\markboth{D.~Finkelshtein, Yu.~Kondratiev, O.~Kutoviy,
E.~Lytvynov}{Binary jumps in continuum. II. Non-equilibrium process}

\section{Introduction}

Let $\Gamma=\Gamma({\R^d})$ denote the space of all locally finite
subsets (configurations) in $\R^d$, $d\in\mathbb N$. A stochastic
dynamics of jumps in continuum is a Markov process on $\Gamma$ in
which groups of particles simultaneously hop over $\R^d$, i.e., at
each jump time several points of the configuration change their
positions.

The simplest case corresponds to the so-called Kawasaki-type
dynamics in continuum. This dynamics is a Markov process on $\Gamma$
in which particles hop over $\R^d$ so that, at each jump time, only
one particle changes its position. For a study of an equilibrium
Kawasaki dynamics in continuum, we refer the reader to the papers
\cite{FKL2007,FKO2009,Glo1982,KKL2008,KLR2007,LO2008,LL2011} and the
references therein. Under the so-called balance condition on the
jump rate, a proper Gibbs distribution is an invariant, and even
symmetrizing measure for such a dynamics, see \cite{Glo1982}. To
obtain a simpler measure, e.g. Poissonian, as a symmetrizing measure
for a Kawasaki-type dynamics, we should either suppose quite
unnatural conditions on the jump rate,  or consider a free dynamics.
In the free Kawasaki dynamics, in the  course of a random evolution,
each particle of the configuration randomly hops over $\mathbb R^d$
without any interaction with other particles (see \cite{KLR2008} for
details).

In the dynamics of binary jumps, at each jump time two points of the (infinite) configuration
change their positions in $\R^d$. A randomness for choosing a
pair of points provides a random interaction between particles of the
system, even in the case where the jump rate only depends on the hopping points
 (and does not depend on the other points of the configuration).
A Poisson measure may be invariant, or even
symmetrizing for such a dynamics. In the first part of the present
paper \cite{FKKL2011a}, we considered such a process with generator
\begin{multline} (LF)(\gamma)=\sum_{\{x_1,\,x_2\}\subset\gamma}\int_{(\R^d)^2}Q( x_1,x_2,dh_1\times dh_2)\\
\times
\big(F(\gamma\setminus\{x_1,x_2\}\cup\{x_1+h_1,x_2+h_2\})-F(\gamma)\big).
\label{genL}
\end{multline}
Here, the measure  $Q( x_1,x_2,dh_1\times dh_2)$ describes the rate
at which two particles, $x_1$ and $x_2$, of configuration $\gamma$
simultaneously hop to $x_1+h_1$ and $x_2+h_2$, respectively. Under
some additional conditions on the measure $Q$, we studied a
corresponding equilibrium dynamics for which a Poisson measure is a
symmetrizing measure. We also considered  two different scalings of
the rate measure $Q$, which   led us to a diffusive dynamics and to
a birth-and-death dynamics, respectively.

In the present paper, we restrict our attention to the special case
of the generator \eqref{genL}. We denote the arrival points by
$y_i=x_i+h_i$, $i=1,2$. An (informal) generator of a binary
jump process of our interest is given by
\begin{multline}\label{tjgen-intro}
\left( LF\right) \left( \gamma \right) =\sum_{\left\{
x_{1},x_{2}\right\} \subset \gamma
}\int_{\mathbb{R}^{d}}\int_{\mathbb{R}^{d}}c\left(
x_1,x_2,y_1,y_2 \right) \\
\times \bigl( F\left( \gamma \setminus \left\{ x_{1},x_{2}\right\}
\cup \left\{ y_{1},y_{2}\right\} \right) -F\left( \gamma \right)
\bigr) dy_{1}dy_{2}.
\end{multline}
Here
\begin{equation}\label{lhgu6r75}c(x_1,x_2,y_1,y_2)=c\left( \left\{ x_{1},x_{2}\right\}
,\left\{y_{1},y_{2}\right\} \right)\end{equation} is a proper non-negative
measurable function. Our aim is to study a non-equilibrium
dynamics corresponding to \eqref{tjgen-intro}. Note that the Poisson
measure with any positive constant intensity will be invariant for
this dynamics. If, additionally,
\begin{equation}\label{utdey7eu}c\left( \left\{
x_{1},x_{2}\right\} ,\left\{y_{1},y_{2}\right\} \right)=c\left(
\left\{y_{1},y_{2}\right\}, \left\{ x_{1},x_{2}\right\}  \right),\end{equation}
then any such measure will even be symmetrizing.

It should be noted that similar dynamics of  finite particle
systems were studied in \cite{BMT1981,BK2003,Kol2006}. In
particular, in \cite{BMT1981}, the authors studied a non-equilibrium
dynamics of velocities of particles, such that the law of
conservation of momentum is satisfied for this system. However, the
methods applied to finite particle systems seem not to be applicable
to infinite systems of our interest.

Let us also note an essential difference between lattice and
continuous systems. An important example of a Markov dynamics on
lattice configurations is the so-called exclusion process. In this
process, particles randomly hop over the lattice under the only
restriction to have no more than one particle at each site of the
lattice. This process may have a Bernoulli measure as an invariant
(and even symmetrizing) measure, but a corresponding stochastic
dynamics has non-trivial properties and possesses an interesting and
reach scaling limit behavior. A straightforward generalization of
the exclusion process to the continuum gives a free Kawasaki
dynamics, because the exclusion restriction (yielding an interaction
between particles)
 obviously disappears for configurations in continuum. To
introduce  the simplest (in certain sense) interaction, we consider the
generator above.

The generator \eqref{tjgen-intro} informally provides a functional
evolution via the (backward) Kolmogorov equation
\begin{equation}
\frac{\partial F_t}{\partial  t}=LF_t,\qquad
F_t\bigm|_{t=0}=F_0.\label{Kolmogor-intro}
\end{equation}
However,  the problem of  existence of a solution to
\eqref{Kolmogor-intro} in some functional space  seems to be a very difficult problem. Fortunately, in
applications, we usually need only an information about a mean value of a
function on $\Ga$ with respect to some probability measure on $\Ga$,
rather than a full point-wise information about this function.
Therefore, we turn to a weak evolution of probability measures
(states) on $\Ga$. This evolution of states is informally given as
a solution to the initial value problem:
\begin{equation}
\frac{d}{d t}\langle F,\mu_t\rangle=\langle LF,\mu_t\rangle , \qquad
\mu_t\bigm|_{t=0}=\mu_0\label{FokkerPlanck-intro},
\end{equation}
provided, of course, that a solution exists. Here $\langle
F,\mu\rangle=\int_\Ga F(\ga)d\mu(\ga)$.

The problem \eqref{FokkerPlanck-intro} may be rewritten in terms of
correlation functionals $k_t$ of states $\mu_t$. The corresponding
dynamics of an infinite vector of correlation functions  has a chain structure which is similar to the BBGKY-hierarchy
for  Hamiltonian dynamics. The corresponding generator of this
dynamics has a two-diagonal (upper-diagonal) structure in a
Fock-type space. Note that the most important case from the point of
view of applications, being also  the most interesting from
the mathematical point of view, is the case of bounded (non-integrable)
correlation functions. Because of  \eqref{FokkerPlanck-intro}, the dynamics of correlation functions should
be treated in a weak form.
Therefore, we consider a pre-dual evolution of the so-called
quasi-observables, whose generator has a low-diagonal structure in a
Fock-type space of integrable functions.

The paper is organized as follows. In Section 2, we describe the
model  and give some necessary preliminary information. Section 3 is
devoted to the functional evolution of both quasi-observables and
correlation functions. We derive some information about the structure and
properties of the two-diagonal generator of the dynamics of
quasi-observables (Propositions~\ref{twodiag} and \ref{opers}). We
further prove a general result for a low-diagonal generator, which
shows that the evolution may be obtained, for all times, recursively
in a scale of Banach spaces whose norms depend on time
(Theorem~\ref{th1}). Next, we construct a dual dynamics on a
finite time interval (Theorem~\ref{th-cr}). We show that this
dynamics is indeed an evolution of correlation functions, which in turn
leads to an evolution of probability measures on $\Ga$
(Theorem~\ref{th-ident}). In Section~4 we consider a Vlasov-type
scaling for our dynamics. The limiting evolution (which exists by
Proposition~\ref{prop-conv-cf}) has a chaos preservation property.
This means that the corresponding dynamics of states transfers a
Poisson measure with a non-homogeneous intensity into a Poisson
measure whose non-homogeneous intensity satisfies a non-linear
evolution (kinetic) equation. We finally present sufficient
conditions for the existence and uniqueness of a solution to this
equation (Proposition~\ref{sol}).

It is worth noting that we rigorously prove the convergence of the
scaled evolution of the infinite particle system to the limiting
evolution (which in turn leads to the kinetic equation). This seems
to be a new step even for finite particle systems.

\section{Preliminaries}
Let  ${\B}({\X})$ be the  Borel $\sigma$-algebra on  ${\X}$,  $d\in\N$, and let  $\Bb$ denote the system of all bounded sets from ${\B}({\X})$.
The configuration space over $\R^d$  is defined as the set
of all locally finite subsets of $\R^d$:
$$\Gamma:=\bigl\{\,\gamma\subset \R^d\bigm|
|\gamma_\Lambda|<\infty\text{ for any }\Lambda\in\Bb\,\bigr\}.$$
Here $|\cdot|$ denotes the cardinality of a set and
$\gamma_\Lambda:= \gamma\cap\Lambda$. One can identify any
$\gamma\in\Gamma$ with the positive Radon measure
$\sum_{x\in\gamma}\delta_x\in{\cal M}(\R^d)$, where  $\delta_x$ is
the Dirac measure with mass at $x$, and  ${\cal M}(\R^d)$
 stands for the set of all
positive  Radon  measures on $\B(\X)$. The space $\Gamma$ can be
endowed with the relative topology as a subset of the space ${\cal
M}(\R^d)$ with the vague topology, i.e., the weakest topology on
$\Gamma$ with respect to which  all maps
$\Gamma\ni\gamma\mapsto\langle f,\gamma\rangle:=\int_{\R^d}
f(x)\,\gamma(dx) =\sum_{x\in\gamma}f(x)$, $f\in C_0(\R^d)$, are
continuous. Here, $C_0(\R^d)$ is the space of all continuous
functions on $\R^d$ with compact support. The corresponding  Borel $\sigma $-algebra $\B(\Ga )$  coincides with the
smallest $\sigma $-algebra on $\Ga$ for which all mappings $\Ga \ni \ga
\mapsto |\ga_ \La |\in{ \N}_0:={\N}\cup\{0\}$ are measurable for any
$\La\in\Bb$, see e.g.\ \cite{AKR1998a}. It is worth noting that $\Ga$ is a Polish space (see e.g.
\cite{KK2006}). 

Let  $\Fc$ denote the set of all measurable
\textit{cylinder functions} on $\Ga$. Each $F\in
\Fc$ is characterized by the
following property: $F(\ga )=F(\ga_\La )$ for some $\La\in
\Bb$ and for any $\ga\in\Ga$.

A stochastic dynamics of binary jumps in continuum is a Markov
process on $\Gamma$ in which pairs of particles simultaneously hop
over $\R^d$, i.e., at each jump time two points of the configuration
change their positions. Thus, an (informal) generator of such a
process has the form \eqref{tjgen-intro}, where $c(x_1,x_2,y_1,y_2)$ is a non-negative measurable function which satisfies \eqref{lhgu6r75} and
$$c(x_1,x_2, \cdot, \cdot)\in L^1_{\mathrm{loc}}(\X\times\X)\quad
\text{for a.a.\ $x_1,x_2\in\X$.}$$
 The function $c$
describes the rate at which pairs of particles hop
over $\X$.

\begin{remark}\label{rem-Ga0}
Note that, in general, the expression on the right hand side of \eqref{tjgen-intro} is not
necessarily well defined for {\it all} $\ga\in\Ga$, even if
$F\in\Fc$. Nevertheless, for such $F$, $\left( LF\right) \left(
\gamma \right)$ has sense at least for all $\gamma\in\Gamma$ with
$|\ga|<\infty$.
\end{remark}

In various applications, the evolution of {\it states of the system} (i.e., {\it measures on
the configuration space\/} $\Gamma$)  helps  one to understand the
behavior of the process and gives possible candidates for invariant states.
In fact, various properties of such an  evolution form the main information
needed in  applications. Using the duality between functions and measures, this evolution may be considered
in a weak form,
given, as usual, by the expression
\begin{equation}\label{dual-F-mu}
\langle F,\mu\rangle=\int_\Ga F(\ga)d\mu(\ga).
\end{equation}
Therefore, the evolution of states is informally given as
a~solution to the initial value problem:
\begin{equation}
\frac{d}{d t}\langle F,\mu_t\rangle=\langle LF,\mu_t\rangle , \qquad
\mu_t\bigm|_{t=0}=\mu_0\label{FokkerPlanck},
\end{equation}
provided, of course, that a solution exists. For a wide class of probability
measures on $\Ga$, one can consider a corresponding  evolution of
their correlation functionals, see below.

The space of $n$-point configurations in an arbitrary $Y\in\B(\X)$
is defined by
\begin{equation*}
\Ga^\n (Y):=\Bigl\{  \eta \subset Y \Bigm| |\eta |=n\Bigr\} ,\quad
n\in { \N}.
\end{equation*}
We set $\Ga^{(0)}(Y):=\{\emptyset\}$. As a set, $\Ga^\n (Y)$ may be
identified with the quotient of $\widetilde{Y^n} := \bigl\{
(x_1,\ldots ,x_n)\in Y^n \bigm| x_k\neq x_l \ \mathrm{if} \ k\neq
l\bigr\}$ with respect to the natural action of the permutation
group $S_n$ on $\widetilde{Y^n}$. Hence, one can introduce the
corresponding Borel $\sigma $-algebra, which will be denoted by
$\B\bigl(\Ga^\n (Y)\bigr)$. The space of finite configurations in an
arbitrary $Y\in\B(\X)$ is defined by
\begin{equation*}
\Ga_0(Y):=\bigsqcup_{n\in {\N}_0}\Ga^\n (Y).
\end{equation*}
This space is equipped with the topology of the disjoint union.
Therefore, we consider the corresponding Borel $\sigma $-algebra $\B
\bigl(\Ga_0(Y)\bigr)$. In the case where $Y=\X$, we will omit the  index
$Y$ in the notation, namely, $\Ga_0=\Ga_{0}(\X)$, $\Ga^\n =\Ga^\n
(\X)$.

The image of the Lebesgue product measure $(dx)^n$ on $\bigl(\Ga^\n
, \B(\Ga^\n )\bigr)$ will be denoted by $m^\n $. We set
$m^{(0)}:=\delta_{\{\emptyset\}}$. Let $z>0$ be fixed. The
Lebesgue--Poisson measure $\la_z$ on $\Ga_0$ is defined by
\begin{equation}  \label{LP-meas-def}
\la_z:=\sum_{n=0}^\infty \frac {z^{n}}{n!}m^\n .
\end{equation}
For any $\La\in\Bb$, the restriction of $\la_z$ to $\Ga
(\La):=\Ga_{0}(\La)$ will also be  denoted by $\la_{z} $. The space
$\bigl( \Ga, \B(\Ga)\bigr)$ is the projective limit of the family of
spaces $\bigl\{\bigl( \Ga(\La), \B(\Ga(\La))\bigr)\bigr\}_{\La \in
\Bb}$. The Poisson measure $\pi_{z}$ on $\bigl(\Ga
,\B(\Ga )\bigr)$ is given as the projective limit of the family of
measures $\{\pi_{z}^\La \}_{\La \in \Bb}$, where $
\pi_{z}^\La:=e^{-zm(\La)}\la_{z} $ is a probability measure on $\bigl(
\Ga(\La), \B(\Ga(\La))\bigr)$ and $m(\La)$ is the Lebesgue measure
of $\La\in \Bb$; for details see e.g.
\cite{AKR1998a}. We will mostly use  the
measure $\la:=\la_1$.

A set $M\in \B (\Ga_0)$ is called bounded if there exist $ \La \in
\Bb$ and $N\in { \N}$ such that $M\subset \bigsqcup_{n=0}^N\Ga^\n
(\La)$. The set of bounded measurable functions with bounded support
will be  denoted by $\Bbs$, i.e., $G\in \Bbs$ if $
G\upharpoonright_{\Ga_0\setminus M}=0$ for some bounded $M\in {\B
}(\Ga_0)$. We also consider  the larger set $L^0_{\mathrm{ls}}(\Ga_0)$ of
all measurable functions on $\Ga_0$ with local support, which means:
$G\in L^0_{\mathrm{ls}}(\Ga_0)$ if $
G\upharpoonright_{\Ga_0\setminus \Ga(\La)}=0$ for some $\La\in\Bb$.
 Any $\B(\Ga_0)$-measurable function $G$ on $ \Ga_0$ is, in
fact, defined by a sequence of functions $\bigl\{G^\n \bigr\}_{n\in{
\N}_0}\,$, where $G^\n $ is a $\B(\Ga^\n )$-measurable function on
$\Ga^\n $. Functions on $\Ga$ and $\Ga_0$ will
be called {\em observables} and {\em quasi-observables},
respectively.

We consider the following mapping from $L^0_{\mathrm{ls}}(\Ga_0)$ into
${{ \mathcal{F}}_{\mathrm{cyl}}}(\Ga )$:
\begin{equation}
(KG)(\ga ):=\sum_{\eta \Subset \ga }G(\eta ), \quad \ga \in \Ga,
\label{KT3.15}
\end{equation}
where $G\in L^0_{\mathrm{ls}}(\Ga_0)$, see, e.g.,
\cite{KK2002,Len1975,Len1975a}. The summation in \eqref{KT3.15} is
taken over all finite subconfigurations $\eta\in\Ga_0$ of the
(infinite) configuration $\ga\in\Ga$; we denote this by the symbol
$\eta\Subset\ga $. The mapping $K$ is linear, positivity preserving,
and invertible with
\begin{equation}
(K^{-1}F)(\eta )=\sum_{\xi \subset \eta }(-1)^{|\eta \setminus \xi
|}F(\xi ),\quad \eta \in \Ga_0.  \label{k-1trans}
\end{equation}

\begin{remark}\label{rem-ext-K-1}
Note that, using formula \eqref{k-1trans}, we can extend the mapping
$K^{-1}$ to functions $F$  which are well-defined, at least, on
$\Ga_0$.
\end{remark}

Let $ \mathcal{M}_{\mathrm{fm}}^1(\Ga )$ denote the set of all
probability measures $\mu $ on $\bigl( \Ga, \B(\Ga) \bigr)$ which
have finite local moments of all orders, i.e., $\int_\Ga |\ga _\La
|^n\mu (d\ga )<+\infty $ for all $\La \in \B_b(\R^{d})$ and $n\in
\N_0$. A measure $\rho $ on $\bigl( \Ga_0, \B(\Ga_0) \bigr)$ is
called locally finite if $\rho (A)<\infty $ for all bounded sets
$A$ from $\B(\Ga _0)$. The set of all such measures is denoted by
$\mathcal{M}_{\mathrm{lf}}(\Ga _0)$.

One can define a transform
$K^{*}:\mathcal{M}_{\mathrm{fm}}^1(\Ga )\rightarrow
\mathcal{M}_{\mathrm{lf}}(\Ga _0),$ which is dual to the
$K$-transform, i.e., for every $\mu \in
\mathcal{M}_{\mathrm{fm}}^1(\Ga )$ and  $G\in \B_{\mathrm{bs}}(\Ga _0)$,
we have
\[
\int_\Ga (KG)(\ga )\mu (d\ga )=\int_{\Ga _0}G(\eta )\,(K^{*}\mu
)(d\eta ).
\]
The measure $\rho _\mu :=K^{*}\mu $ is called the {\it correlation measure
of\/} $\mu $.

As shown in \cite{KK2002}, for any $\mu \in
\mathcal{M}_{\mathrm{fm}}^1(\Ga )$ and  $G\in L^1(\Ga _0,\rho
_\mu )$, the series \eqref{KT3.15} is $\mu $-a.s. absolutely
convergent. Furthermore, $KG\in L^1(\Ga ,\mu )$ and
\begin{equation}
\int_{\Ga _0}G(\eta )\,\rho _\mu (d\eta )=\int_\Ga (KG)(\ga )\,\mu
(d\ga ). \label{Ktransform}
\end{equation}

A measure $\mu \in \mathcal{M}_{\mathrm{fm} }^1(\Ga )$ is called
locally absolutely continuous with respect to $\pi:=\pi_1$ if
$\mu^\La :=\mu \circ p_\La ^{-1}$ is absolutely continuous with
respect to $\pi^\La:=\pi_1^\La $ for all $\La \in \Bb$. Here
$\Ga\ni\ga\mapsto p_\La(\ga):=\ga\cap\La\in\Ga(\La)$. In this case, the correlation measure
$\rho _\mu  $ is absolutely continuous with respect to
$\la=\la_1 $. A {\it correlation functional of\/} $\mu$ is then defined by
\[
k_{\mu}(\eta):=\frac{d\rho_{\mu}}{d\la}(\eta),\quad
\eta\in\Ga_{0}.
\]
The functions $k_\mu^{(0)}:=1$ and
\begin{equation}
k_{\mu}^\n :(\R^{d})^{n}\longrightarrow\R_{+}, \quad n\in\N,
\end{equation}
\[
k_{\mu}^\n (x_{1},\ldots,x_{n}):=
\begin{cases}
k_{\mu}(\{x_{1},\ldots,x_{n}\}), & \mbox{if $(x_{1},\ldots,x_{n})\in
\widetilde{(\R^{d})^{n}}$}\\ 0, & \mbox{otherwise}
\end{cases}
\]
are called {\it correlation functions of\/} $\mu$, and they are well known in statistical
physics, see e.g \cite{Rue1969}, \cite{Rue1970}.

In view of Remark~\ref{rem-Ga0} and \eqref{k-1trans}, the mapping
\[
(\widehat{L}G)(\eta):= (K^{-1}LKG)(\eta), \qquad \eta\in\Ga_0,
\]
is well defined for any $G\in\Bbs$, where $K^{-1}$ is understood in
the sense of Remark~\ref{rem-ext-K-1}.

Let $k$ be a measurable function on $\Ga_0$
such that $\int_Bk\,d\la<\infty$ for any bounded
$B\in\B(\Ga_0)$. Then, for any $G\in\Bbs$,
we can consider an analog of pairing
\eqref{dual-F-mu},
\begin{equation}\label{dual-sm}
\left\langle \!\left\langle G,\,k\right\rangle \!\right\rangle
:=\int_{\Gamma _{0}}G\cdot k\,d\lambda.
\end{equation}
Using this duality, we may consider a dual mapping $\widehat{L}^*$
of $\widehat{L}$. As a result, we obtain two initial value problems:
\begin{align}
&&\frac{\partial G_t}{\partial t}&=\widehat{L}G_t,&&
G_t\bigm|_{t=0}=G_0,&&\label{Kolmogor-q}
\\
&&\frac{d }{d t}\left\langle \!\left\langle G,\,k_t\right\rangle
\!\right\rangle=\langle \!\langle G,\,\widehat{L}^*k_t\rangle
\!\rangle&=\langle \!\langle \widehat{L}G,\,k_t\rangle \!\rangle,&&
k_t\bigm|_{t=0}=k_0.&&\label{Kolmogor-k}
\end{align}
We are going to solve the first problem in a proper functional space
over $\Ga_0$. The second problem, corresponding to \eqref{FokkerPlanck},
can be realized by means of \eqref{dual-sm}, or solved
independently.

The next section is devoted to a (rigorous) solution of
these two problems.

\section{Functional evolution}

We denote, for any $n\in \mathbb N$,
\begin{equation}  \label{spaceXn}
X_{n}:=L^{1}\bigl( ( \mathbb{R}^{d}) ^{n},dx^{\left( n\right) }
\bigr) ,
\end{equation}
where we set $dx^{\left( n\right) }=d x_{1}\dotsm dx_{n} $ and  $X_{0}:=
\mathbb{R}$. The symbol $\|\cdot\|_{X_n}$ stands for the norm of the space
\eqref{spaceXn}.

For an arbitrary $C>0$, we consider the functional Banach space
\begin{equation}
\mathcal{L}_{C}:=L^{1}\bigl(\Gamma _{0},C^{|\eta |}d\lambda (\eta )\bigr).
\label{spaceLC}
\end{equation}
Throughout this paper, the symbol $\left\Vert \cdot \right\Vert _{\L
_C}$ denotes for the norm of the space \eqref{spaceLC}. Then, for
any $G\in\L _C$, we may identify $G$ with the sequence $(G^\n
)_{n\geq0}\,$, where $G^\n $ is a symmetric function on
$(\mathbb{R}^d)^n$ (defined almost everywhere) such that
\begin{equation}\label{normLc}
\|G\|_{\L _C}=\sum_{n=0}^\infty \frac{1}{n!}\int_{(\mathbb{R}^d)^n}
\bigl|G^\n (x^\n )\bigr|C^n dx^\n =\sum_{n=0}^\infty\frac{C^n}{n!}
\bigl\|G^\n \bigr\|_{X_n}<\infty .
\end{equation}
In particular, $G^\n \in X_n$, $n\in\mathbb N_0$. Here and below we set
$x^{\left( n\right) }=\left( x_{1},\ldots ,x_{n}\right)$.

We consider the dual space $(\mathcal{L}_{C})^{\prime }$ of
$\mathcal{L}_{C}$. It is evident that this space can be realized as
the Banach space
\begin{equation*}
{\mathcal{K}}_{C}:=\bigl\{k:\Gamma _{0}\rightarrow {\mathbb{R}}\,\bigm|
k\cdot C^{-|\cdot |}\in L^{\infty }(\Gamma _{0},d\lambda )\bigr\}
\end{equation*}
with the norm $\Vert k\Vert _{{\mathcal{K}}_{C}}:=\Vert k(\cdot ) C^{-|\cdot
|}\Vert _{L^{\infty }(\Gamma _{0},\lambda )}$, where  the
pairing between any $G\in\mathcal{L} _{C}$ and
$k\in{\mathcal{K}}_{C}$ is given by \eqref{dual-sm}. In particular,
$$\left\vert \left\langle \!\left\langle G,k\right\rangle
\!\right\rangle \right\vert \leq \Vert G\Vert _{C}\cdot \Vert k\Vert _{{
\mathcal{K}}_{C}}.$$ Clearly, $k\in {\mathcal{K}}_{C}$ implies
\begin{equation}\label{ineq5}
|k(\eta )|\leq \Vert k\Vert _{{\mathcal{K}}_{C}}\,C^{|\eta |}\qquad \mathrm{
for}\ \lambda \text{-a.a.}\ \eta \in \Gamma _{0}.
\end{equation}

\begin{proposition}\label{twodiag}
Let for a.a. $x_1,x_2,y_1\in\X$
\begin{equation}\label{c-tilde}
\tilde{c}(x_1,x_2,y_1)=\tilde{c}(\{x_1,x_2\},y_1):=\int_{\mathbb{R}
^{d}}c\left(  x_{1},x_{2}, y_{1},y_{2} \right) dy_{2}<\infty.
\end{equation}
Then, for any $G\in B_{\mathrm{bs}}(\Gamma_0)$, the following formula holds
\begin{equation}  \label{descentpairjump}
(\widehat{L}G)\left( \eta \right) =\left( L_{0}G\right) \left( \eta
\right) +\left( WG\right) \left( \eta \right),
\end{equation}
where
\begin{align}
\left( L_{0}G\right) \left( \eta \right) &=\sum_{\left\{ x_{1},x_{2}\right\}
\subset \eta }\int_{\mathbb{R}^{d}}\int_{\mathbb{R}^{d}}c\left(
x_{1},x_{2}, y_{1},y_{2}\right) \nonumber  \\
&\quad\times \bigl( G\left( \eta \setminus \left\{ x_{1},x_{2}\right\} \cup
\left\{ y_{1},y_{2}\right\} \right) -G\left( \eta \right) \bigr)
dy_{1}dy_{2},  \label{L0} \\
\left( WG\right) \left( \eta \right)
&=\sum_{x_{2}\in \eta }\sum_{x_{1}\in \eta \setminus x_{2}}\int_{\mathbb{R}
^{d}}\tilde{c}\left( x_{1},x_{2} ,y_{1}\right) \nonumber  \\
&\quad\times\bigl( G\left( \left( \eta \setminus x_{2}\right)
\setminus x_{1}\cup y_{1}\right) -G\left( \eta \setminus
x_{2}\right) \bigr) dy_{1}. \label{W}
\end{align}
\end{proposition}

\begin{proof}
We have, for $F=KG$,
\begin{align*}
&\left( KG\right) \left( \gamma \setminus \left\{ x_{1},x_{2}\right\} \cup
\left\{ y_{1},y_{2}\right\} \right) -\left( KG\right) \left( \gamma \right)
\\
&=\sum_{\eta \Subset \gamma \setminus \left\{ x_{1},x_{2}\right\} \cup
\left\{ y_{1},y_{2}\right\} }G\left( \eta \right) -\sum_{\eta \Subset \gamma
}G\left( \eta \right) \\
&=\sum_{\eta \Subset \gamma \setminus \left\{ x_{1},x_{2}\right\} }G\left(
\eta \cup y_{1}\right) +\sum_{\eta \Subset \gamma \setminus \left\{
x_{1},x_{2}\right\} }G\left( \eta \cup y_{2}\right)  \\
&\quad+\sum_{\eta \Subset
\gamma \setminus \left\{ x_{1},x_{2}\right\} }G\left( \eta \cup y_{1}\cup
y_{2}\right)-\sum_{\eta \Subset \gamma \setminus \left\{ x_{1},x_{2}\right\} }G\left(
\eta \cup x_{1}\right)  \\
&\quad-\sum_{\eta \Subset \gamma \setminus \left\{
x_{1},x_{2}\right\} }G\left( \eta \cup x_{2}\right) -\sum_{\eta \Subset
\gamma \setminus \left\{ x_{1},x_{2}\right\} }G\left( \eta \cup x_{1}\cup
x_{2}\right) \\
&=\left( K\left( G\left( \cdot \cup y_{1}\right) +G\left( \cdot \cup
y_{2}\right) +G\left( \cdot \cup \left\{ y_{1},y_{2}\right\}
\right)\right.\right. \\
&\qquad\left.\left. -G\left( \cdot \cup x_{1}\right) -G\left( \cdot \cup
x_{2}\right) -G\left( \cdot \cup \left\{ x_{1},x_{2}\right\} \right) \right)
\right) \left( \gamma \setminus \left\{ x_{1},x_{2}\right\} \right) .
\end{align*}
Hence,  for any measurable $h$ on $\Gamma \times \mathbb{R}^{d}\times \mathbb{R}
^{d}$,
\begin{align*}
&K^{-1}\left( \sum_{\left\{ x_{1},x_{2}\right\} \subset \cdot }h\left( \cdot
\setminus \left\{ x_{1},x_{2}\right\} ,x_{1},x_{2}\right) \right) \left(
\eta \right) \\
&=\sum_{\xi \subset \eta }\left( -1\right) ^{\left\vert \eta \setminus \xi
\right\vert }\sum_{\left\{ x_{1},x_{2}\right\} \subset \xi }h\left( \xi
\setminus \left\{ x_{1},x_{2}\right\} ,x_{1},x_{2}\right) \\
&=\sum_{\left\{ x_{1},x_{2}\right\} \subset \eta }\,\,\sum_{\xi
\subset \eta \setminus \left\{ x_{1},x_{2}\right\} }\left( -1\right)
^{\left\vert \eta \setminus \left\{ x_{1},x_{2}\right\} \setminus
\xi \right\vert }h\left( \xi
,x_{1},x_{2}\right) \\
&=\sum_{\left\{ x_{1},x_{2}\right\} \subset \eta }\left( K^{-1}h\left( \cdot
,x_{1},x_{2}\right) \right) \left( \eta \setminus \left\{
x_{1},x_{2}\right\} \right) .
\end{align*}
Therefore,
\begin{align*}
&(\widehat{L}G)\left( \eta \right) =\left( K^{-1}LKG\right) \left( \eta \right) \\
&=\sum_{\left\{ x_{1},x_{2}\right\} \subset \eta }\int_{\mathbb{R}^{d}}\int_{
\mathbb{R}^{d}}c\left( \left\{ x_{1},x_{2}\right\} ,\left\{
y_{1},y_{2}\right\} \right)  \bigl( G\left( \eta \setminus \left\{ x_{1},x_{2}\right\} \cup
y_{1}\right) \\&\qquad\qquad +G\left( \eta \setminus \left\{ x_{1},x_{2}\right\} \cup
y_{2}\right) +G\left( \eta \setminus \left\{ x_{1},x_{2}\right\} \cup
\left\{ y_{1},y_{2}\right\} \right) \bigr) dy_{1}dy_{2} \\
&-\sum_{\left\{ x_{1},x_{2}\right\} \subset \eta }\int_{\mathbb{R}^{d}}\int_{
\mathbb{R}^{d}}c\left( \left\{ x_{1},x_{2}\right\} ,\left\{
y_{1},y_{2}\right\} \right) \bigl( G\left( \eta \setminus \left\{ x_{1},x_{2}\right\} \cup
x_{1}\right) \\&\qquad\qquad+G\left( \eta \setminus \left\{ x_{1},x_{2}\right\} \cup
x_{2}\right) +G\left( \eta \setminus \left\{ x_{1},x_{2}\right\} \cup
\left\{ x_{1},x_{2}\right\} \right) \bigr) dy_{1}dy_{2} ,
\end{align*}
from where the statement follows.
\end{proof}

As we noted before, any function $G$ on $\Ga_0$ may be identified
with the infinite vector $\bigl( G^{\left( n\right) }\bigl) _{n\geq
0}$ of symmetric functions. Due to this identification, any operator
on functions on $\Ga_0$ may be considered as an infinite operator
matrix. By Proposition~\ref{twodiag}, the operator matrix $\widehat
L$ has a two-diagonal structure. More precisely, on the main
diagonal, we have operators $L_{0}^{\left( n\right) }$, $n\in\mathbb
N_0$, where $L_{0}^{\left( 0\right) }=L_{0}^{(1)}=0$ and for $n\geq
2$
\begin{multline}
\bigl(L_{0}^{\left( n\right) }G^{\left( n\right)
}\bigr)(x^\n)=\sum_{i=1}^n\sum_{j=i+1}^n\int_{\mathbb{R}^{d}}\int_{\mathbb{R}^{d}}c\left(
x_{i},x_{j}, y_{1},y_{2},\right)  \\
\times \Bigl( G^\n\bigl( x_1,\ldots,\underset{\overset{\wedge}{i}}{y_1},\ldots,\underset{\overset{\wedge}{j}}{y_2},\ldots,x_n\bigr) -G^\n\bigl( x^\n\bigr) \Bigr)
dy_{1}dy_{2}.\label{Ln}
\end{multline}
On the low diagonal, we have operators $W^\n$, $n\in\mathbb N$,
where $W^{(1)}=0$ and for $n\geq 2$
\begin{align}
\bigl( W^\n G^{(n-1)}\bigr) \bigl( x^\n \bigr)
=&\,2\sum_{i=1}^n\sum_{j=i+1}^n\int_{\mathbb{R}
^{d}}\tilde{c}\left( x_{i},x_{j}
,y_{1}\right)\nonumber\\&\times\Bigl( G^{(n-1)}\bigl(
x_1,\ldots,x_{i-1},x_{i+1},\ldots,\underset{\overset{\wedge}{j}}{y_1},\ldots,x_n\bigr)\nonumber
\\&\qquad-G^{(n-1)}\bigl(
x_1,\ldots,x_{i-1},x_{i+1},\ldots,x_n\bigr) \Bigr) dy_{1}.\label{Wn}
\end{align}

Let us formulate our main conditions on the
rate $c$:
\begin{align}
c_1:&=\esssup_{x_{1},x_2\in \mathbb{R}^{d}}\int_{\mathbb{R}^{d}}\int_{\mathbb{R}
^{d}}c\left( x_{1},x_{2}, y_{1},y_{2} \right) dy_{1}dy_{2}<\infty
, \label{c-1}\\ c_2:&=\esssup_{x_{1},x_2\in \mathbb{R}^{d}}\int_{\mathbb{R}^{d}}\int_{\mathbb{R}
^{d}}c\left( y_{1},y_{2},
x_{1},x_{2} \right) dy_{1}dy_{2}<\infty, \label{c-2}\\
c_3:&=\esssup_{x_1\in \mathbb{R}^{d}}\int_{\mathbb{R}^{d}}\int_{\mathbb{R}^{d}}\int_{\mathbb{R}
^{d}}c\left( x_1,x_{2}, y_{1},y_{2} \right) dy_{1}dy_{2}dx_2<\infty
, \label{c-3}\\
c_4:&=\esssup_{x_1\in \mathbb{R}^{d}}\int_{\mathbb{R}^{d}}\int_{\mathbb{R}^{d}}\int_{\mathbb{R}
^{d}}c\left( y_{1},y_{2}, x_{1},x_{2} \right)
dy_{1}dy_{2}dx_2<\infty . \label{c-4}
\end{align}
Under  conditions \eqref{c-1}--\eqref{c-2}, we define the
following functions
\begin{align}\label{a1}
a_1(x_1,x_2):&=\int_{\mathbb{R}^{d}}\int_{\mathbb{R}
^{d}}c\left( x_{1},x_{2},
y_{1},y_{2} \right) dy_{1}dy_{2}\in[0,\infty),\\
a_2(x_1,x_2):&=\int_{\mathbb{R}^{d}}\int_{\mathbb{R}
^{d}}c\left( y_{1},y_{2}, x_{1},x_{2} \right)
dy_{1}dy_{2}\in[0,\infty).\label{a2}
\end{align}

\begin{remark} Note that, if the function $c$ satisfies the symmetry condition \eqref{utdey7eu}, then conditions \eqref{c-2}, \eqref{c-4} follow from \eqref{c-1}, \eqref{c-3}, respectively, and
\begin{equation}\label{huytew4u6}
a_1(x_1,x_2)=a_2(x_1,x_2).\end{equation}
In this case, the operator
$L_0^\n$ is symmetric in $L^2\bigl((\X)^n, dx^\n\bigr)$. Moreover,
the operator $L$ given by \eqref{tjgen-intro} is (informally) symmetric in
$L^2(\Ga,\pi_z)$ for any $z>0$.
\end{remark}

\begin{proposition}\label{opers}
{\rm (i)}~Let \eqref{c-1}, \eqref{c-2} hold. Then, for any $G^\n\in
X_n$,
\begin{equation}\label{estLn}
\|L_0^\n G^\n\|_{X_n}\leq \frac{n(n-1)}{2}(c_1+c_2)\|G^\n\|_{X_n}.
\end{equation}
Moreover, if additionally
\begin{equation}\label{dominate}
a_2(x_1,x_2)\leq a_1(x_1,x_2) \quad \mathrm{for
\ a.a.}\
x_1,x_2\in
\X,
\end{equation}
then $L_0^\n$ is the generator of a contraction
semigroup in $X_n$.

{\rm (ii)}~Let \eqref{c-3}, \eqref{c-4} hold. Then, for any
$G^{(n-1)}\in X_{n-1}$,
\begin{equation}\label{estWn}
\|W^\n G^{(n-1)}\|_{X_n}\leq n(n-1)(c_3+c_4)\|G^{(n-1)}\|_{X_{n-1}}.
\end{equation}
\end{proposition}

\begin{proof}
The estimates \eqref{estLn}, \eqref{estWn}
follow directly from \eqref{Ln}, \eqref{Wn}
and \eqref{c-1}--\eqref{c-4}. To prove that
the bounded operator $L_0^\n$ is the generator
of a contraction semigroup in $X_n$, it is
enough to show that $L_0^\n$ is dissipative
(see e.g. \cite{LF1961}).
For any $\kappa>0$ and   $G^\n\in
X_n$,
\begin{align*}
&\|L_0^\n G^\n -\kappa G^\n\|_{X_n}\\=&\,\int_{(\X)^n} \Biggl|
\sum_{i=1}^n\sum_{j=i+1}^n\int_{\mathbb{R}^{d}}\int_{\mathbb{R}^{d}}c\left(
x_{i},x_{j}, y_{1},y_{2},\right)G^\n\bigl( x_1,\ldots,\underset{\overset{\wedge}{i}}{y_1},\ldots,\underset{\overset{\wedge}{j}}{y_2},\ldots,x_n\bigr)dy_{1}dy_{2}   \label{Ln}\\
&\qquad-\sum_{i=1}^n\sum_{j=i+1}^n a_{1}(x_i,x_j) G^\n\bigl( x^\n\bigr) -\kappa G^\n\bigl( x^\n\bigr)
\Biggr|
dx^\n,\\\intertext{ and, using the obvious inequality $\|f-g\|_{X_n}\geq\bigl|\|f\|_{X_n}-\|g\|_{X_n}\bigr|$,
we  continue}\geq&\, \Biggl|\int_{(\X)^n}
\biggl| \sum_{i=1}^n\sum_{j=i+1}^n\int_{\mathbb{R}^{d}}\int_{\mathbb{R}^{d}}c\left(
x_{i},x_{j}, y_{1},y_{2},\right)\\&\qquad\times
G^\n\bigl( x_1,\ldots,\underset{\overset{\wedge}{i}}{y_1},\ldots,\underset{\overset{\wedge}{j}}{y_2},\ldots,x_n\bigr)dy_{1}dy_{2}\biggr|dx^\n\\
&\qquad\qquad-\int_{(\X)^n}\biggl(\sum_{i=1}^n\sum_{j=i+1}^na_{1}(x_i,x_j)+\kappa\biggr) \Bigl|G^\n\bigl( x^\n\bigr)\Bigr|
dx^\n\Biggr|.
\end{align*}
Since $G^\n$ is a symmetric function, we get
\begin{align*}
&\int_{(\X)^n}
\biggl| \sum_{i=1}^n\sum_{j=i+1}^n\int_{\mathbb{R}^{d}}\int_{\mathbb{R}^{d}}c\left(
x_{i},x_{j}, y_{1},y_{2},\right)\\&\qquad\times G^\n\bigl( x_1,\ldots,\underset{\overset{\wedge}{i}}{y_1},\ldots,\underset{\overset{\wedge}{j}}{y_2},\ldots,x_n\bigr)dy_{1}dy_{2}\biggr|dx^\n\\
\leq&\,\frac{n(n-1)}{2} \int_{(\X)^n} a_2(y_1,y_2)\Bigl| G^\n\bigl(
y_1,y_2,x_1,\dots,x_{n-2}\bigr)\Bigr|dy_{1}dy_{2}dx_1\ldots
dx_{n-2}\,.
\end{align*}
Therefore, if \eqref{dominate} holds, then
\[
\|L^\n_0 G^\n -\kappa G^\n\|_{X_n}\geq \kappa \|G^\n\|_{X_n},
\]
which proves the dissipativity of $L^\n_0$, see e.g.\
\cite[Proposition 3.23]{EN2000}.
\end{proof}

\begin{remark} If the function $c$ satisfies the symmetry condition \eqref{utdey7eu}, then  \eqref{dominate} trivially follows from \eqref{huytew4u6}.
 \end{remark}

In Theorems \ref{th1} and \ref{th-cr} below, we formulate general
results which are applicable to our dynamics under assumptions
\eqref{c-1}--\eqref{c-4}, \eqref{dominate}.

\begin{theorem}\label{th1}
Consider the initial value problem
\begin{equation}
\begin{aligned}
\dfrac{\partial }{\partial t}G_{t}\left( \eta \right) &=\left(
L_{0}G_{t}\right) \left( \eta \right) +\left( WG_{t}\right) \left(
\eta \right),
\quad t>0,\ \eta\in\Ga_0, \\[2mm]
G_{t}\bigm\vert _{t=0}&=G_{0}.
\end{aligned}
\label{Cauchy}
\end{equation}
Here, for any $G=\left( G^{\left( n\right) }\right) _{n\geq 0}\,$,
\begin{align*}
\left( L_{0}G\right) ^{\left( n\right) } &=L_{0}^{\left( n\right) }G^{\left(
n\right) },~n\geq 1; \\
\left( WG\right) ^{\left( n\right) } &=W^{\left( n\right) }G^{\left(
n-1\right) },~n\geq 2; \\
(L_0G)^{(0)}&=(WG)^{(0)}=(WG)^{(1)}=0.
\end{align*}
Further suppose  that $L_{0}^{\left( n\right) }$ is a bounded
generator of a strongly
continuous contraction semigroup $e^{tL_{0}^{\left( n\right) }}$ in $X_{n}$,
while $W^{\left( n\right) }$ is a bounded operator from $X_{n-1}$ into $X_{n}$
whose norm satisfies
\begin{equation*}
\bigl\Vert W^{\left( n\right) }\bigr\Vert _{X_{n-1}\rightarrow X_{n}}\leq
Bn\left( n-1\right) ,\quad n\geq 1,
\end{equation*}
for some $B\geq 1$ which is independent of $n$.

Let $C>0$ and  $G_{0 }\in \L_C$. Then the initial value problem (\ref{Cauchy})\ has a unique solution $G_t\in\L_{\rho(t,C)}$, where
\begin{equation}\label{weight}
\rho(t,C):=\frac{C}{1+BCt}\,.
\end{equation}
Furthermore,
\begin{equation}  \label{maincompar}
\|G_t\|_{\L_{\rho(t,C)}}\leq\|G_0 \|_{\L_{C}}.
\end{equation}
\end{theorem}

\begin{proof}
Let us rewrite the initial value problem (\ref{Cauchy}) as an
infinite system of differential equations. Namely, for any $n\geq1$,
\begin{equation}\label{dfdsg}
\dfrac{\partial }{\partial t}G_{t}^\n\bigl( x^\n\bigr) =\bigl( L_{0}^\n G_{t}^\n \bigr) \bigl( x^\n\bigr)
+\bigl( W^\n G_{t}^{(n-1)}\bigr) \bigl( x^\n\bigr)
\end{equation}
with  $G_{t}^{\left( 0\right) } =G_{0 }^{\left( 0\right) }$.
This system may be solved recurrently:
for each $n\geq 1$
\begin{align}
G_{t}^{\left( n\right) }\bigl( x^{\left( n\right) }\bigr) &=\left(
e^{tL_{0}^{\left( n\right) }}G_{0}^{\left( n\right) }\right) \bigl(
x^{\left( n\right) }\bigr) \nonumber\\&\quad+\int_{0 }^{t}\left( e^{\left( t-s\right)
L_{0}^{\left( n\right) }}W^{\left( n\right) }G_{s}^{\left( n-1\right)
}\right) \bigl( x^{\left( n\right) }\bigr)ds. \label{exprforsol}
\end{align}
Iterating \eqref{exprforsol},
we obtain
\[
G_t^\n(x^\n)=\sum_{k=0}^n \bigl(V_{k,n}(t)G_0^{(n-k)}\bigr)(x^\n),
\]
where $V_{k,n}(t):X_{n-k}\rightarrow X_n$ is
given by
\begin{align*}
V_{k,n}(t):=&\int_0^t\int_0^{s_1}\dotsm\int_0^{s_{k-1}}
e^{(t-s_1)L_{0}^\n }W^\n e^{(s_1-s_2)L_{0}^{(n-1)}} W^{(n-1)}\ldots
\\&\qquad \times
e^{(s_{k-1}-s_{k})L_{0}^{(n-k+1)}}W^{(n-k+1)}e^{s_kL_{0}^{(n-k)}}ds_k\dotsm
ds_1
\end{align*}
for $2\leq k\leq n-1$, and
\begin{align*}
V_{1,n}(t)G^{(n-1)}&:=\int_0^t e^{(t-s_1)L_{0}^\n }W^\n e^{s_1L_{0}^{(n-1)}}G^{(n-1)}ds_1,\\
V_{0,n}(t)G^\n &:=e^{tL_{0}^\n }G^\n ,\\
V_{n,n}(t)G^\n&:=\chi_{\{n=0\}}G^{(0)}.
\end{align*}
Then, using the conditions on $L_0$ and $W$, for $1\leq k\leq n-1$,
$n\geq2$
\begin{align*}
\bigl\|V_{k,n}(t)G^{(n-k)}\bigr\|_{X_n}&\leq\frac{t^k}{k!}B^k
n(n-1)(n-1)(n-2)\dotsm(n-k+1)(n-k)\\&=(tB)^k\frac{n!}{k!(n-k)!}\frac{(n-1)!}{(n-k-1)!}\bigl\|G^{(n-k)}\bigr\|_{X_{n-k}},
\end{align*}
and $\bigl\|V_{0,n}(t)G^\n \bigr\|_{X_n}\leq\|G^\n\|_{X_n}$ for
$n\geq1$. Therefore, for $n\geq 1$,
\begin{align}
\bigl\|G_t^\n\bigr\|_{X_n}&\leq \sum_{k=0}^{n-1}(tB)^k\frac{n!}{k!(n-k)!}\,\frac{(n-1)!}{(n-k-1)!}\bigl\|G^{(n-k)}\bigr\|_{X_{n-k}}
\nonumber\\
&=\sum_{k=1}^{n}(tB)^{(n-k)}\frac{n!}{(n-k)!k!}\,\frac{(n-1)!}{(k-1)!}\bigl\|G_{0}^{(k)}\bigr\|_{X_{k}}.\label{estforGtn}
\end{align}

Then, for any $q(t)>0$,
\begin{align*}
&\bigl\Vert G_{t}\bigr\Vert _{\L_{Cq(t)}} \\
&=\bigl\vert G_{t}^{\left( 0\right) }\bigr\vert +\sum_{n=1}^{\infty }\frac{
C^{n}q^n(t)}{n!}\,\bigl\Vert G_{t}^{\left( n\right) }\bigr\Vert _{X_{n}} \\
&\leq \bigl\vert G_{0}^{\left( 0\right) }\bigr\vert +\sum_{n=1}^{\infty }
\frac{C^{n}q^n(t)}{n!}\sum_{k=1}^{n}\bigl\Vert G_{0}^{\left(
k\right)
}\bigr\Vert _{X_{k}}(tB)^{n-k}\frac{n!}{k!\left( n-k\right) !}\,\frac{
\left( n-1\right) !}{\left( k-1\right) !} \\
&=\bigl\vert G_{0}^{\left( 0\right) }\bigr\vert +\sum_{k=1}^{\infty
}\bigl\Vert G_{0}^{\left( k\right) }\bigr\Vert
_{X_{k}}\frac{1}{k!}\sum_{n=k}^{\infty
}C^{n}q^n(t)(tB)^{n-k}\frac{1}{\left( n-k\right)
!}\,\frac{\left(
n-1\right) !}{\left( k-1\right) !} \\
&=\bigl\vert G_{0}^{\left( 0\right) }\bigr\vert +\sum_{k=1}^{\infty
}\bigl\Vert G_{0}^{\left( k\right) }\bigr\Vert _{X_{k}}\frac{C^{k}q^{k}(t)}{k!}\sum_{n=0}^{\infty }q^n(t)(tBC)^{n}\frac{\left( n+k-1\right) !}{
n!\left( k-1\right) !}\,.
\end{align*}
Now, let $q(t)=\dfrac{1}{1+BCt}$. For any $x\in[0,1)$ and  $m\in\N$,
\[
\Bigl(\frac{1}{1-x}\Bigr)^{m+1}=\sum_{n=0}^\infty x^n\,\frac{\left(
n+m\right) !}{n!m !}.
\]
Applying this equality to $x=q(t)BCt<1$ and $m=k-1$, we obtain
\[
q^{k}(t)\sum_{n=0}^{\infty }q^n(t)(tBC)^{n}\frac{\left( n+k-1\right) !}{
n!\left( k-1\right) !}= \biggl(\frac{q(t)}{1-q(t)BCt}\biggr)^k=1.
\]
Therefore,
\[
\bigl\Vert G_{t}\bigr\Vert _{\L_{Cq(t)}}\leq \bigl\vert G_{0}^{\left( 0\right) }\bigr\vert
+\sum_{k=1}^{\infty
}\bigl\Vert G_{0}^{\left( k\right) }\bigr\Vert _{X_{k}}\frac{C^{k}}{
k!}= \|G_0\|_{\L_C},
\]
which proves the  statement.
\end{proof}

In fact, we have a linear evolution operator
$$V(t):\L_C\rightarrow\L_{\rho(t,C)},$$
satisfying $G_t=V(t)G_0$ and
\begin{equation}\label{Vtdef}
\|V(t)\|_{\L_C\rightarrow\L_{\rho(t,C)}}\leq 1
\end{equation}

\begin{theorem}\label{th-cr}
Let the conditions of Theorem~\ref{th1} be satisfied. Further suppose
that there exists $A>0$ such that $\|L_0^\n\|_{X_n\rightarrow
X_n}\leq An(n-1)$, $n\geq1$. Let $C_0>0$, $k_0\in\K_{C_0}$,
$T=\frac{1}{BC_0}$. Then for any $t\in(0,T)$, there exists
 $k_t\in\K_{C_t}$ with
\begin{equation}\label{Ct}
C_t=\dfrac{C_0}{1-BC_0t},
\end{equation}
such that, for any $G\in\Bbs$,
\[
\frac{d}{d t}\langle\!\langle G,k_t\rangle\!\rangle=\langle\!\langle
(L_0+W)G,k_t\rangle\!\rangle, \quad t\in(0,T).
\]
Moreover, for any $t\in(0,T)$,
\begin{equation}\label{strongcontr}
\|k_t\|_{\K_{C_t}}\leq\|k_0\|_{\K_{C_0}}.
\end{equation}
\end{theorem}

\begin{proof}
Let $t\in(0,\,T)$ be arbitrary. The function
$f(x)=\rho(t,x)=\frac{x}{1+xBt}$, $x\geq0$, increases to
$\frac{1}{tB}$ as $x\rightarrow +\infty$. Since $C_0<\frac{1}{tB}$,
there exists a unique solution to $f(x)=C_0$, namely, $x=C_t$, given
by \eqref{Ct}. Take any $G_0\in\L_{C_t}$. By Theorem~\ref{th1},
there exists an evolution $G_0\mapsto G_\tau$ for any $\tau>0$ such
that $G_\tau\in \L_{\rho(\tau,C_t)}$. Consider this evolution at the
moment $\tau=t$. Since
\begin{equation}\label{rho_t}
\rho(t,C_t)=\frac{C_t}{1+BC_tt}=C_0,
\end{equation}
we have $G_t\in\L_{C_0}$. Therefore, $\langle\!\langle G_t,
k_0\rangle\!\rangle$ is well-defined. Moreover, by
\eqref{maincompar},
\begin{align}\nonumber
\bigl|\langle\!\langle G_t, k_0\rangle\!\rangle\bigr|&\leq
\|G_t\|_{\L_{C_0}}\|k_0\|_{\K_{C_0}}=
\|G_t\|_{\L_{\rho(t,C_t)}}\|k_0\|_{\K_{C_0}}\\&\leq
\|G_0\|_{\L_{C_{t}}}\|k_0\|_{\K_{C_0}}.\label{bd}
\end{align}
Therefore, the mapping $G_0\mapsto\langle\!\langle G_t,
k_0\rangle\!\rangle$ is a linear continuous functional on the space
$\L_{C_{t}}$. Hence, there exists $k_t\in\K_{C_{t}}$ such that, for
any $G_0\in\L_{C_t}$,
\begin{equation}\label{wq23}
\langle\!\langle G_0, k_t\rangle\!\rangle=\langle\!\langle G_t,
k_0\rangle\!\rangle=\langle\!\langle V(t)G_0, k_0\rangle\!\rangle.
\end{equation}
We note  that $k_t$ depends on $k_0$ and does not depend on $G_0$.
Further, \eqref{bd} implies \eqref{strongcontr}.

Let now $G_0\in \Bbs$. Consider a function $g=g_{G_0,k_0}:
  [0,T)\rightarrow
  \R$, $g(t):=\langle\!\langle G_t, k_0\rangle\!\rangle=\langle\!\langle G_0, k_t\rangle\!\rangle$. We have
\begin{equation}\label{ser1}
g(t)=\langle\!\langle G_t, k_0\rangle\!\rangle=
\sum_{n=0}^\infty\frac{1}{n!}\int_{(\X)^n}G_t^\n\bigl(x^\n\bigr)k_0^\n\bigl(x^\n\bigr)dx^\n.
\end{equation}
By \eqref{bd}, for any $[0,T']\subset[0,T)$,
\[
\bigl|\langle\!\langle G_t, k_0\rangle\!\rangle\bigr| \leq
\|G_0\|_{\L_{C_{T'}}}\|k_0\|_{\K_{C_0}}, \quad t\in[0,T'].
\]
Hence, the series \eqref{ser1} converges on $[0,T']$. Using  the
well-known representation
\begin{equation}\label{represent}
e^{tL_0^\n}G_0=G_0+\int_0^t e^{sL_0^\n}L_0^\n G_0ds
\end{equation}
(see e.g.\ \cite[Lemma 1.3 (iv)]{EN2000}), we derive from
\eqref{exprforsol} and Fubini's theorem:
\begin{align}
g_n(t):&=\int_{(\X)^n}G_t^\n\bigl(x^\n\bigr)k_0^\n\bigl(x^\n\bigr)dx^\n\nonumber\\
&=\int_{(\X)^n}\Bigl(e^{tL_0^\n}G_0^\n\Bigr)\bigl(x^\n\bigr)k_0^\n\bigl(x^\n\bigr)dx^\n\nonumber\\
&\quad+\int_0^t\int_{(\X)^n}\Bigl(e^{(t-s)L_0^\n}W^\n
G_s^{(n-1)}\Bigr)\bigl(x^\n\bigr)k_0^\n\bigl(x^\n\bigr)dx^\n
ds\nonumber\\&=\int_{(\X)^n}G_0^\n\bigl(x^\n\bigr)k_0^\n\bigl(x^\n\bigr)dx^\n\nonumber\\&\quad+
\int_0^t\int_{(\X)^n}\Bigl(e^{sL_0^\n}L_0^\n G_0^\n\Bigr)\bigl(x^\n\bigr)k_0^\n\bigl(x^\n\bigr)dx^\n ds\nonumber\\
&\quad+\int_0^t\int_{(\X)^n}\Bigl(e^{(t-s)L_0^\n}W^\n
G_s^{(n-1)}\Bigr)\bigl(x^\n\bigr)k_0^\n\bigl(x^\n\bigr)dx^\n
ds.\label{gnt}
\end{align}
Since $L_0^\n:X_n\rightarrow X_n$ and $W^\n:X_{n-1}\rightarrow X_n$
are bounded, the functions inside the time integrals are continuous in
$s$. Therefore, by \eqref{gnt} and \eqref{exprforsol}, $g_n(t)$~is
differentiable on $(0,T)$ and
\begin{align}
g_n'(t)&=\int_{(\X)^n}\Bigl(L_0^\n e^{tL_0^\n}
G_0^\n\Bigr)\bigl(x^\n\bigr)k_0^\n\bigl(x^\n\bigr)dx^\n\nonumber\\&\quad+\int_0^t\int_{(\X)^n}\Bigl(L_0^\n
e^{(t-s)L_0^\n}W^\n
G_s^{(n-1)}\Bigr)\bigl(x^\n\bigr)k_0^\n\bigl(x^\n\bigr)dx^\n
ds\nonumber\\
&\quad+\int_{(\X)^n}\Bigl(W^\n
G_t^{(n-1)}\Bigr)\bigl(x^\n\bigr)k_0^\n\bigl(x^\n\bigr)dx^\n
ds\nonumber\\&=\int_{(\X)^n}\bigl(L_0^\n G_t^\n
\bigr)\bigl(x^\n\bigr)k_0^\n\bigl(x^\n\bigr)dx^\n
ds\nonumber\\
&\quad+\int_{(\X)^n}\Bigl(W^\n
G_t^{(n-1)}\Bigr)\bigl(x^\n\bigr)k_0^\n\bigl(x^\n\bigr)dx^\n
ds.\label{ex2}
\end{align}
Hence, for any $n\geq2$,
\[
|g'_n(t)|\leq \|k_0\|_{\K_{C_0}}C_0^nn(n-1)\bigl(A\|G_t\|_{X_n}+B\|G_t\|_{X_{n-1}}\bigr).
\]
Analogously to the proof of Theorem~\ref{th1}, we obtain, for all
$t\in[0,T']\subset[0,T)$,
\begin{align}
&\quad\sum_{n=1}^\infty\frac{1}{n!}|g'_n(t)|\nonumber\\&\leq\mathrm{const}\cdot
\sum_{n=1}^\infty\frac{1}{n!}C_0^nn(n-1)\sum_{k=1}^{n}(tB)^{(n-k)}\frac{n!}{(n-k)!k!}\frac{(n-1)!}{(k-1)!}\bigl\|G_{0}^{(k)}\bigr\|_{X_{k}}\nonumber\\
&\leq \mathrm{const}\cdot
\sum_{k=1}^\infty\frac{1}{k!}\bigl\|G_{0}^{(k)}\bigr\|_{X_{k}}\sum_{n=k}^{\infty}C_0^n(T'B)^{(n-k)}\frac{n(n-1)}{(n-k)!}\frac{(n-1)!}{(k-1)!}\nonumber\\&=
\mathrm{const}\cdot
\sum_{k=1}^\infty\frac{C_0^k}{k!}\bigl\|G_{0}^{(k)}\bigr\|_{X_{k}}
\sum_{n=0}^{\infty}C_0^n(T'B)^{n}\frac{(n+k-1)}{n!}\frac{(n+k)!}{(k-1)!}<\infty,
\label{ex4}
\end{align}
since $G_0\in\Bbs$ (and so  there exists $K\in\N$ such that $G_0^{(k)}=0$
for all $k\geq K$) and the inner series converges
as $C_0 T'B<1$.

Hence, $g(t)$ is differentiable on any $[0,T']\subset[0,T)$. Next,
\eqref{ser1}, \eqref{ex2}, and \eqref{ex4} imply that
\begin{equation}\label{eq123}
g'(t)=\frac{d}{d t}\left\langle \!\left\langle
G_t,\,k_0\right\rangle \!\right\rangle=\langle \!\langle
(L_0+W)G_{t},\,k_0\rangle\!\rangle
\end{equation}
and, moreover, $(L_0+W)G_{t}\in\L_{C_0}$. Therefore, using
\eqref{wq23}, \eqref{eq123} and the obvious inclusion
$(L_0+W)G_0\in\L_{C_t}$, we obtain
\begin{align*}
\frac{d }{d t}\left\langle \!\left\langle G_0,\,k_t\right\rangle
\!\right\rangle&=\frac{d }{d t}\left\langle \!\left\langle
G_t,\,k_0\right\rangle \!\right\rangle=\langle \!\langle
(L_0+W)G_{t},\,k_0\rangle\!\rangle=\langle \!\langle
(L_0+W)V(t)G_{0},\,k_0\rangle\!\rangle\\&=\langle \!\langle
V(t)(L_0+W)G_{0},\,k_0\rangle\!\rangle=\langle
\!\langle(L_0+W)G_{0},\,k_t\rangle\!\rangle,
\end{align*}
provided
\begin{equation}\label{eqnMMM}
(L_0+W)V(t)G_{0}= V(t)(L_0+W)G_{0}.
\end{equation}

To prove \eqref{eqnMMM}, we consider, for each $N\in\mathbb N$, the
space $\mathbb{X}_N:=\bigoplus\limits_{n=0}^N X_n$ with the norm
$\|\cdot\|_{\mathbb{X}_N}:=\sum\limits_{n=0}^N\|\cdot\|_{{X_n}}$.
For any $\mathbb{G}\in\mathbb{X}_N$,
$\mathbb{G}=(\mathbb{G}^{(0)},\ldots,\mathbb{G}^{(N)})$, we define
the following function on $\Ga_0$:
$$\mathbb{I}_N\mathbb{G}:=(\mathbb{G}^{(0)},\ldots,\mathbb{G}^{(N)},0,0,\ldots).$$
For any function $G$ on $\Ga_0$, $G=(G^{(0)},\ldots, G^\n,\ldots)$,
with $G^\n\in X_n$, we define the following element of
$\mathbb{X}_N$:
$$\mathbb{P}_NG:=(G^{(0)},\ldots, G^\n).$$ The~system of differential
equations \eqref{dfdsg} for $1\leq n\leq N$ can be considered as one
equation $\frac{\partial }{\partial t}\mathbb{G}_t=\mathbb{L}_N
\mathbb{G}_t$ in $\mathbb{X}_N$ with
\[
\mathbb{L}_N:=\mathbb{P}_N(L_0+W)\mathbb{I}_N.
\]
Clearly, $\mathbb{L}_N$ is a bounded operator in $\mathbb{X}_N$.
Hence, there exists a unique vector-valued solution of this
equation, $\mathbb{G}_t=e^{t\mathbb{L}_N}\mathbb{G}_0$. The 
 $n$-th
component   of $\mathbb{G}_t$, i.e., $\mathbb{G}_t^{(n)}$,
 coincides with the $G_t^{(n)}$
obtained in Theorem~\ref{th1}, for each $0\leq n\leq N$, where
$G_0=\mathbb{I}_N\mathbb{G}_0$. More precisely, for $0\leq n\leq N$,
\begin{equation}\label{wewq}
(V(t)G_0)^\n=(\mathbb{I}_Ne^{t\mathbb{L}_N
}\mathbb{G}_0)^\n=(e^{t\mathbb{L}_N
}\mathbb{G}_0)^\n=(e^{t\mathbb{L}_N }\mathbb{P}_NG_0)^\n.
\end{equation}
It is well known that a bounded operator $\mathbb{L}_N$ commutes
with its semigroup $e^{t\mathbb{L}_N }$. Note also that, for $0\leq
n\leq N$, $ G^\n=(\mathbb{P}_NG)^\n$. Therefore, for all $N\geq1$,
$0\leq n\leq N$, and for $G_{0}=\mathbb{I}_N\mathbb{G}_0$, we obtain
\begin{align*}
((L_0+W)V(t)G_{0})^{(n)} &=(\mathbb{P}_N(L_0+W)V(t)G_{0})^{(n)}
=(\mathbb{L}_N e^{t\mathbb{L}_N }\mathbb{G}_0)^{(n)}\\&=
(e^{t\mathbb{L}_N }\mathbb{L}_N \mathbb{G}_0)^{(n)} =
(e^{t\mathbb{L}_N }\mathbb{P}_N(L_0+W)\mathbb{I}_N
\mathbb{G}_0)^{(n)} \\&=(V(t)(L_0+W)G_{0})^{(n)},
\end{align*}
where in the last equality we applied \eqref{wewq} for
$(L_0+W)G_{0}$ instead of $G_0$. Hence, \eqref{eqnMMM} holds. 
\end{proof}

\begin{remark}
Note that  the initial value problem
$\dfrac{\partial}{\partial t}k_t=\widehat{L}^*k_t$,
$k_t\bigr|_{t=t_1}=k_{t_1}\in\K_{C_{t_1}}$ for some
$t_1<T=\dfrac{1}{BC_0}$,   has a solution only on the time
interval $[t_1,t_1+T_1)=[t_1,T)$, since
$T_1=\dfrac{1}{BC_{t_1}}=\dfrac{1-BC_0t_1}{BC_{0}}=T-t_1$.
\end{remark}

\begin{remark} Using an estimate analogous to  \eqref{ex4}, one  can
show that $\dfrac{\partial{G_t}}{\partial t}\in\L_{\rho(t,C)}$ if
$G_0$ belongs to $\Bbs$ (or even to a larger subset of $\L_C$).
\end{remark}

Thus, by Theorems \ref{th1} and \ref{th-cr}, under conditions
\eqref{c-1}--\eqref{c-4}, \eqref{dominate} for the binary jumps
dynamics with generator \eqref{tjgen-intro} we have the evolution of
quasi-observables and the corresponding dual one. We will now show
that the latter evolution generates an evolution of probability
measures on $\Ga$.

\begin{theorem}\label{th-ident}
Let \eqref{c-1}--\eqref{c-4} and  \eqref{dominate} hold. Fix a measure
$\mu\in\mathcal{M}^1_\mathrm{fm}(\Ga)$ which has a correlation
functional $k_\mu\in\K_{C_0}$, $C_0>0$. Consider the evolution
$k_\mu\mapsto k_t\in\K_{C_t}$, $t\in(0,T)$, where
$T=1/(c_3+c_4)C_0$. Then, for any $t\in(0,T)$, there exists a unique
measure $\mu_t\in\mathcal{M}^1_\mathrm{fm}(\Ga)$ such that $k_t$ is
the  correlation functional of $\mu_t$.
\end{theorem}

\begin{proof}
We first recall the following definition. Let a measurable, non-negative function
$k$ on $\Ga_0$ be such that $\int_M k(\eta)\, d\la(\eta)<\infty$  for any bounded $M\in {\B }(\Ga_0)$. The function $k$
is said to be \textit{Lenard positive definite} if $\ll G,k\rr\geq0$
for any $G\in\Bbs$ such that $KG\geq0$. It was
shown in \cite{Len1975a} that any such $k$ is the correlation
functional of some probability measure on $\Ga$. If,
additionally, $k\in\K_C$ for some $C>0$, then
this measure is uniquely defined (cf.\ \cite{Len1973}) and
belongs to $\mathcal{M}_{\mathrm{fm} }^1(\Ga )$
(cf.\ \cite{KK2002}). Therefore, to prove the theorem, it
is enough to show that $k_t$ is Lenard positive definite for any
$t\in(0,T)$.

Since the measure $\mu \in \mathcal{M}_{\mathrm{fm} }^1(\Ga )$ has
the correlation functional $k_\mu$, $\mu$ is locally absolutely
continuous with respect to $\pi$, and for any $\La\in\Bbs$ and
$\la$-a.a $\eta\in\Ga(\La)$
\begin{equation}\label{inv}
\dfrac{d\mu^{\La}}{d\la}(\eta)=\int_{\Ga(\La)}(-1)^{|\xi|}k_\mu(\eta\cup\xi)d\la(\xi),
\end{equation}
see \cite[Proposition 4.3]{KK2002}. Since $k_\mu\in\K_{C_0}$, we
have, by \eqref{inv}, \eqref{ineq5}, and \eqref{LP-meas-def},
\begin{equation}\label{est21}
\frac{d\mu^{\La}}{d\la}(\eta)\leq
\|k_\mu\|_{\K_{C_0}}e^{C_0m(\La)}C_0^{|\eta|}
\end{equation}
for $\la$-a.a $\eta\in\Ga(\La)$.

We fix $\La_0\in\Bb$ and consider the projection
$\mu_0:=\mu^{\La_0}$ on $\Ga(\La_0)$. 
%Clearly, $\mu_0$ may be
%considered as a measure on the whole of $\Ga$ if we set $\mu_0$ 
%equal to $0$ outside of  $\Ga(\La_0)$. 
 By \eqref{est21}, for
$\la$-a.a. $\eta\in\Ga(\La_0)$
\begin{equation}\label{expr1}
R_0(\eta):=\dfrac{d\mu_0}{d\la}(\eta)\leq
A_0C_0^{|\eta|},
\end{equation}
where $A_0:=\|k_\mu\|_{\K_{C_0}}e^{C_0m(\La_0)}$. 
Clearly, $\mu_0$ may be
considered as a measure on the whole of $\Ga$ if we set $R_0$ 
to be equal to $0$ outside of  $\Ga(\La_0)$. Hence,
$$
\mu_{0}(A)=\int_{\Ga(\La_{0})\cap A}R_{0}({\eta})d\la(\eta),\quad A\in \mathcal{B}(\Ga).
$$
%Evidently,  $R_0(\eta)=0$ if $\eta\notin\Ga(\La_0)$. 
On the other hand, $R_{0}$ being extended by zero outside of $\Ga(\La_0)$ can also be regarded as a $\mathcal{B}(\Ga_{0})$-measurable function. Evidently, that in this case $0 \leq R_0\in L^1(\Ga_0,d\la)$ with $\int_{\Ga_0}R_0d\la=1$. Note that
\begin{equation}\label{kmu0}
k_0:=1\!\!1_{\Ga(\La_0)}k_\mu \in\K_{C_0}
\end{equation}
 is the correlation functional of
$\mu_0$. Here and below, $1\!\!1_\Delta$ stands for the indicator
function of a set $\Delta$. By
 \cite[Proposition~4.2]{KK2002},  for $\la$-a.a. $\eta\in\Ga(\La_0)$
 \begin{equation}\label{expr2}
 k_{0}(\eta)=\int_{\Ga(\La_0)} R_{0}(\eta\cup\xi)d\la(\xi).
 \end{equation}

There exists an $N_0\in\N$ such that
$\int_{(\X)^{N_0}}R_0^{(N_0)}dx^{(N_0)}>0$ (otherwise $R_0=0$
$\la$-a.e.). We set
\[
r:=\int_{\bigsqcup_{n=0}^{N_0} \Ga^\n}R_0(\eta)d\la(\eta)\in(0,1].
\]
For each $N\geq N_0$, we define
\begin{equation}\label{expr1.5}
R_{0,N}(\eta)=1\!\!1_{\{|\eta|\leq N\}}(\eta) R_0(\eta)\biggl(
\int_{\bigsqcup_{n=0}^N \Ga^\n}R_0(\eta)d\la(\eta)\biggr)^{-1}.
\end{equation}
Then, clearly, $0\leq R_{0,N}\in L^1(\Ga_0,d\la)$, with
$\int_{\Ga_0}R_{0,N}\,d\la=1$.  Moreover, $R_{0,N}$ has a bounded
support on $\Ga_0$. By \eqref{expr1.5} and \eqref{expr1} we have
\begin{equation}\label{expr22}
R_{0,N}(\eta)\leq r^{-1} R_0(\eta)\leq r^{-1}A_0C_0^{|\eta|}
\end{equation}
for $\la$-a.a. $\eta\in\Ga_0$.

We define a probability measure
$\mu_{0,N}\in\mathcal{M}^1_\mathrm{fm}(\Ga)$, concentrated on
$\Ga_0$, by $d\mu_{0,N}=R_{0,N}d\la$. By
\cite[Proposition~4.2]{KK2002}, the correlation functional $k_{0,N}$
of $\mu_{0,N}$ has the following representation
\begin{equation}\label{expr3}
k_{0,N}(\eta)=\int_{\Ga(\La_0)} R_{0,N}(\eta\cup\xi)d\la(\xi)
\end{equation}
for $\la$-a.a. $\eta\in\Ga(\La_0)$. It is evident now that $k_{0,N}$
has a bounded support on $\Ga_0$. Moreover, by \eqref{expr2},
\eqref{expr3}, and the first inequality in \eqref{expr22}, we get
$k_{0,N}\leq\frac{1}{r}k_0\in\K_{C_0}$ and
\begin{equation}\label{normestN}
\|k_{0,N}\|_{\K_{C_0}}\leq \frac{1}{r}\,\|k_{0}\|_{\K_{C_0}}.
\end{equation}
By the definition of a correlation functional, for any $G\in\Bbs$
\begin{align}
\bigl| \ll G,k_0\rr-\ll G,k_{0,N}\rr\bigr|=&\,
\biggl|\int_{\Ga_0}(KG)(\eta)\bigl(R_0(\eta)-R_{0,N}(\eta)\bigr)d\la(\eta)\biggr|
\nonumber\\\leq&\,D\int_{\Ga(\La_0)}(1+|\eta|)^M\bigl|R_0(\eta)-R_{0,N}(\eta)\bigr|d\la(\eta)\label{expr4}
\end{align}
for some $D=D(G)>0$ and $M=M(G)\in\N$ (see
\cite[Proposition~3.1]{KK2002}). By \eqref{expr1.5},
$R_{0,N}(\eta)\rightarrow R_{0}(\eta)$ for $\la$-a.a.
$\eta\in\Ga(\La_0)$. Furthermore, by \eqref{expr1} and \eqref{expr22},
\[
\bigl|R_0(\eta)-R_{0,N}(\eta)\bigr|\leq A_0(1+r^{-1})C_0^{|\eta|}.
\]
By \eqref{LP-meas-def},
\[
\int_{\Ga(\La_0)}(1+|\eta|)^M C_0^{|\eta|}d\la(\eta)<\infty.
\]
Therefore, by the dominated convergence theorem, \eqref{expr4}
yields
\begin{equation}\label{expr5}
 \lim_{N\rightarrow\infty}\ll G,k_{0,N}\rr=\ll G,k_0\rr.
\end{equation}

As before, we identify a function $F$ on~$\Ga_0$ with a sequence of
symmetric functions $F^\n$ on $(\X)^n$, $n\in\N_0$.
Fix any $G\in \Bbs$ and let $F$ be the restriction of $KG$ to $\Gamma_0$. Then there
exist $\La=\La_F\in\Bb$, $M=M_F\in\N$, and $D=D_F>0$ such that for
all $\eta\in\Ga_0$, $$|F(\eta)|=|F(\eta\cap\La)|\leq D
\bigl(1+|\eta\cap\La|)^M$$ (see \cite[Proposition~3.1]{KK2002}). In
particular, $F^\n$ is bounded on  $(\R^{d})^n$ for each $n$. We restrict the
operator $L$ given by \eqref{tjgen-intro} to functions on $\Ga_0$. This
restriction, $L_0$, is given by \eqref{L0}.

We define, for any $R\in L^1(\Ga_0,\la)$, the function $L_0^\ast R$
on $\Ga_0$ by $(L_0^\ast R)^\n:=(L_0^\n)^* R^\n$, where $(L_0^\n)^*$
is given by the right hand side of \eqref{Ln} in which
$c(x_i,x_j,y_1,y_2)$ is replaced by $c(y_1,y_2,x_i,x_j)$.
Analogously to the proof of Proposition~\ref{opers}, we conclude that  $(L_0^\n)^*$ is a bounded
generator of a strongly continuous semigroup on
$X_n=L^1\bigl((\X)^n, dx^\n\bigr)$. In the dual space
$\quad X_n^*:=$ $L^\infty\bigl((\X)^n, dx^\n\bigr)$, we consider the dual
operator to $(L_0^\n)^*$, which is just the $L_0^\n$ given by
\eqref{Ln}. It is easy to see that, under  condition \eqref{c-1},
$L_0^\n$ is a bounded operator on $X_n^*$. Note that $L_0^\n1=0$
implies
\[
\int_{(\X)^n} e^{t(L_0^\n)^*}R^\n dx^\n=\int_{(\X)^n} \bigl(e^{t
L_0^\n}1\bigr)R^\n dx^\n=\int_{(\X)^n} R^\n dx^\n.
\]

To show that $e^{t(L_0^\n)^*}$ preserves the cone $X_n^+$ of all
positive functions in $X_n$, we write $(L_0^\n)^*=L_1+L_2$, where
\begin{multline*}
\bigl(L_{1}R^\n
\bigr)(x^\n)=\sum_{i=1}^n\sum_{j=i+1}^n\int_{\mathbb{R}^{d}}\int_{\mathbb{R}^{d}}c\left(
 y_{1},y_{2},x_{i},x_{j}\right)  \\
\times  R^\n\bigl(
x_1,\ldots,\underset{\overset{\wedge}{i}}{y_1},\ldots,\underset{\overset{\wedge}{j}}{y_2},\ldots,x_n\bigr)
dy_{1}dy_{2}
\end{multline*}
and
\[
\bigl(L_{2}R^{\left( n\right)
}\bigr)(x^\n)=-\Biggl(\sum_{i=1}^n\sum_{j=i+1}^n a_2(x_i,x_j)\Biggr)
R^\n(x^\n).
\]
Clearly, $L_1$ and $L_2$ are bounded operators on $X_n$. Since
$c\geq0$, $L_1$ preserves the cone $X_n^+$, hence so does the
semigroup $e^{tL_1}$. Since $L_2$ is a bounded multiplication
operator, the semigroup $e^{tL_2}$ is a positive multiplication
operator in $X_n$. Therefore, $e^{t(L_0^\n)^*}$ preserves $X_n^+$ by
the Lie--Trotter product formula.

Therefore, if we define functions $R_{t,N}$, $N\geq N_0$ (cf. (\ref{expr1.5})) on $\Ga_0$ 
by $$R_{t,N}^\n:=e^{t(L_0^\n)^*}R^\n_{0,N},$$ then $0\leq R_{t,N}\in
L^1(\Ga_0,d\la)$ with $\int_{\Ga_0}R_{t,N}d\la=1$. Note that
$R_{t,N}^\n\equiv0$ for $n>N$. Therefore, we can define a measure
$\tilde{\mu}_{t,N}\in\mathcal{M}^1_\mathrm{fm}(\Ga)$, concentrated
on $\Ga_0$, by $d\tilde{\mu}_{t,N}=R_{t,N} d\la$ (in fact, the
measure $\tilde{\mu}_{t,N}$ is concentrated on $\bigsqcup_{n=0}^{N}
\Ga^\n$). We denote by $\tilde{k}_{t,N}$ the correlation functional
of $\tilde{\mu}_{t,N}$.

For each function $G$ on $\Ga_0$ we define
$K_0G:=(KG)\!\!\upharpoonright _{\Ga_0}$. Take any $G_0\in\Bbs$ such
that $KG_0\geq0$ on $\Ga$. We denote $F_0:=K_0G_0\geq0$ on $\Ga_0$.
We have, by the definition of a correlation functional,
\begin{equation}\label{expr7}
\ll G_0, k_{t,N}\rr=\langle K_0G_0,R_{t,N}\rangle\geq0.
\end{equation}
On the other hand, if we define a function $U(t)F_0$ on $\Ga_0$ by
$(U(t)F_0)^\n= e^{tL_0^\n}F_0^\n$, we obtain
\begin{align}\nonumber
&\langle F_0, R_{t,N}\rangle=\sum_{n=0}^{N}\frac{1}{n!}\langle
F_0^\n, R_{t,N}^\n\rangle=\sum_{n=0}^{N}\frac{1}{n!}\langle F_0^\n,
e^{t(L_0^\n)^*}R_{0,N}^\n\rangle\\=&\,
\sum_{n=0}^{N}\frac{1}{n!}\langle e^{tL_0^\n}F_0^\n,
R_{0,N}^\n\rangle=\langle U(t)F_0,
R_{0,N}\rangle\nonumber\\=&\,\langle\!\langle
K_0^{-1}U(t)K_0G_0,k_{0,N}\rangle\!\rangle.\label{expr8}
\end{align}
It is evident, by Proposition~\ref{twodiag}, that
\begin{equation}\label{descsg}
(K_0^{-1}U(t)K_0G_0)^\n=e^{t (L_0^\n+W^\n)}G_0^\n=(V(t)G_0)^{\n},
\end{equation}
where $V(t)$ is as in \eqref{Vtdef}.

As a result, from \eqref{expr7}--\eqref{descsg}, we get
\begin{equation}\label{expr9}
\langle\!\langle G_0,k_{t,N}\rangle\!\rangle=\langle\!\langle
G_t,k_{0,N}\rangle\!\rangle,
\end{equation}
where $G_t=V(t)G_0$.
By \eqref{kmu0}, $k_0\in \K_{C_0}$. Hence, by Theorem~\ref{th-cr},
for any $t\in(0,T)$, there exists $\tilde{k}_t\in\K_{C_t}$ such that
\begin{equation}\label{expr10}
\langle\!\langle G_0,\tilde{k}_t\rangle\!\rangle=\langle\!\langle
G_t,k_0\rangle\!\rangle.
\end{equation}
Note that here, for a given $t\in(0,T)$, we may consider
$G_0\in\Bbs\subset \L_{C_t}$, where $C_t$ is given by \eqref{Ct}.
Then, by the proof of Theorem~\ref{th-cr}, $G_t=V(t)G_0\in\L_{C_0}$.
By \eqref{expr9} and \eqref{expr10}, to prove that
\begin{equation}\label{expr11}
\ll G_0,\tilde{k}_t\rr=\lim_{N\rightarrow \infty}\ll
G_0,k_{t,N}\rr,
\end{equation}ı
we only need to show that
\begin{equation}\label{expr12}
\lim_{N\rightarrow\infty}\langle\!\langle
G_t,k_{0,N}\rangle\!\rangle =\langle\!\langle
G_t,k_{0}\rangle\!\rangle.
\end{equation}
The latter fact is a direct consequence of \eqref{expr5} if we take
into account that $G_t\in\L_{C_0}$ and that the set $\Bbs$ is dense
in $\L_{C_0}$. Indeed, let us consider, for a fixed $t\in(0,T)$ and
for any $\eps>0$, a function $G\in\Bbs$ such that
$\|G-G_t\|_{\L_{C_0}}<\eps$. Then, by \eqref{expr5}, there exists an
$N_1\geq N_0$, such that, for any $N\geq N_1$
$$
\bigl| \ll G,k_{0,N}\rr-\ll G,k_{0}\rr\bigr|<\eps.
$$
Therefore, by \eqref{normestN}, for any $N\geq N_1$
\begin{align*}
&\,\bigl| \ll G_t,k_{0,N}\rr-\ll
G_t,k_{0}\rr\bigr|\\\leq&\,\|G-G_t\|_{\L_{C_0}}\|k_{0,N}\|_{\K_{C_0}}+\bigl|
\ll G,k_{0,N}\rr-\ll G,k_{0}\rr\bigr|
+\|G-G_t\|_{\L_{C_0}}\|k_{0}\|_{\K_{C_0}}\\<&\,\eps (r^{-1}+1)
\|k_0\|_{\K_{C_0}}+\eps,
\end{align*}
which proves \eqref{expr12}. Therefore, \eqref{expr11} holds. Hence,
by \eqref{expr7}, $\tilde{k}_t=:\tilde{k}_t^{\La_0}$ is Lenard positive
definite for any $t\in(0,T)$.

As a result, for each  $\La\in\Bb$, the evolution
$k_{\mu^\La}\mapsto\tilde{k}_t^{\La}$, $t\in(0,T)$, preserves
positive-definiteness and $\ll G_0,\tilde{k}_t^\La \rr = \ll G_t,
k_{\mu^\La} \rr$. On the other hand, by Theorem~\ref{th-cr}, we have
the evolution $k_\mu\mapsto k_t$, $t\in(0,T)$ satisfying $ \ll
G_0,k_t \rr = \ll G_t, k_\mu \rr$.

Since $k_{\mu^\La}=1\!\!1_{\Ga(\La)}k_\mu$, it is evident that, for
any $t\in(0,T)$, $\ll G_t, k_{\mu^\La} \rr\rightarrow \ll G_t, k_\mu
\rr$ as $\La\nearrow\X$. Therefore, $$\ll G_0,k_t
\rr=\lim_{\La\nearrow\X} \ll G_0,\tilde{k}_t^\La \rr \geq0.$$ Hence,
for each $t\in(0,T)$, there exists a unique measure
$\mu_t\in\mathcal{M}^1_\mathrm{fm}(\Ga)$ whose correlation
functional is $k_t$.
\end{proof}

\section{Vlasov-type scaling}

For the reader's convenience, we start with explaining the idea of
the Vlasov-type scaling. A general scheme for both birth-and-death and conservative dynamics may be found in
\cite{FKK2010a}. Certain realizations of this approach were studied
in \cite{FKK2010b, FKK2010c, FKKoz2011, BKKK2011}.

We would like to construct a scaling of the generator $L$, say
$L_\eps$, $\eps>0$, such that the following requirements  are satisfied. Assume
 we have an evolution $V_\eps(t)$ corresponding to the equation
$\frac{\partial}{\partial t} G_{t,\eps}=\widehat{L}_\eps
G_{t,\eps}$. Assume that,  in a proper functional space, we have the
dual evolution $V^*_\eps(t)$ with respect to the duality
\eqref{dual-sm}. Let us choose an initial function for this dual
evolution with a big singularity in $\eps$, namely,
$k_0^{(\eps)}(\eta) \sim \eps^{-|\eta|} r_0(\eta)$ as
$\eps\rightarrow 0$ ($\eta\in\Ga_0$), with some function $r_0$,
being independent of $\eps$. Our first requirement on the scaling $L\mapsto
L_\eps$ is that the evolution $V^*_\eps(t)$ preserve the order of
the singularity:
\begin{equation}\label{ordersing}
\bigl(V^*_\eps(t) k_0^{(\eps)}\bigr)(\eta) \sim \eps^{-|\eta|}
r_t(\eta) \quad \text{as }\eps\rightarrow 0, \ \ \eta\in\Ga_0,
\end{equation}
where $r_t$ is such that the dynamics $r_0 \mapsto r_t$ preserves
the so-called Lebesgue--Poisson exponents. Namely, if
$r_0(\eta)=e_\la(p_0,\eta):=\prod_{x\in\eta}p_0(x)$, then
$r_t(\eta)=e_\la(p_t,\eta)=\prod_{x\in\eta}p_t(x)$. (Here,
$\prod_{x\in\emptyset}:=1$.) Furthermore, we require that the
$p_t$'s satisfy a  (nonlinear, in general) differential equation
\begin{equation}\label{V-eqn-gen}
\dfrac{\partial}{\partial t}p_t(x) = \upsilon(p_t)(x),
\end{equation}
which  will be called a {\it Vlasov-type equation}.

For any $c>0$, we set $\left( R_{c}G\right) \left( \eta \right)
=c^{\left\vert \eta \right\vert }G\left( \eta \right) $. Roughly
speaking, \eqref{ordersing} means that
$$R_\eps V^*_\eps(t)R_{\eps^{-1}}
r_0 \sim  r_t.$$
This gives us a hint to consider the map $R_\eps
\widehat{L}^*_\eps R_{\eps^{-1}} $, which is dual to
\[
\widehat{L}_{\varepsilon ,\mathrm{ren}}=R_{\varepsilon ^{-1}}\widehat{L}
_{\varepsilon }R_{\varepsilon }.
\]

\begin{remark} We expect that the limiting dynamics
for $R_\eps V^*_\eps(t)R_{\eps^{-1}}$ will preserve the
Lebesgue--Poisson exponents. Note that the Lebesgue--Poisson
exponent $e_\la(p_t)$, $t\geq0$, is the correlation functional of
the Poisson measure $\pi_{p_t}$ on $\Ga$ with the non-constant
intensity $p_t$ (for a rigorous definition of such a Poisson
measure, see e.g.\ \cite{AKR1998a}). Therefore, at least
heuristically, we expect to have the limiting dynamics:
$\pi_{p_0}\mapsto\pi_{p_t}$, where $p_t$ satisfies the equation
\eqref{V-eqn-gen}.
\end{remark}

Below we will realize this scheme in the case of the generator  $L$ given by
\eqref{tjgen-intro}. We will consider, for any $\varepsilon >0$, the
scaled operator $L_\eps=\eps L$. Then, obviously,
$\widehat{L}_{\varepsilon }G=\varepsilon \widehat{L}G=\varepsilon
L_{0}G+\varepsilon WG$, where $L_0$ and $W$ are given by
Proposition~\ref{twodiag}.

\begin{proposition}
For any $\eps>0$
\begin{equation}\label{renexpr}
\widehat{L}_{\varepsilon ,\mathrm{ren}}=\eps L_0+W.
\end{equation}
Moreover, let \eqref{c-1}--\eqref{c-4}, \eqref{dominate} hold.
Then, for any $
C>0$, the initial value problem
\begin{equation}\label{Cauchy-eps}
\begin{aligned}
\dfrac{\partial }{\partial t}G_{t,\varepsilon }\left( \eta \right)
&=\bigl( \widehat{L}_{\varepsilon ,\mathrm{ren}}G_{t,\varepsilon
}\bigr) \left( \eta
\right), \\[2mm]
G_{t,\varepsilon }\bigm\vert _{t=0}&=G_{0,\varepsilon }\in \L_C
\end{aligned}
\end{equation}
has a unique solution  $G_{t,\varepsilon }\in\L_{\rho(t,C)}$, where $\rho(t,C)$ is given by \eqref{weight}.
\end{proposition}
\begin{proof}
By Proposition~\ref{twodiag}, $L_{0}R_{\varepsilon } =R_{\varepsilon }L_{0}$ and $WR_{\varepsilon } =\varepsilon ^{-1}R_{\varepsilon }W$. Therefore,
\begin{align*}
\widehat{L}_{\varepsilon ,\mathrm{ren}}&=R_{\varepsilon
^{-1}}\varepsilon L_{0}R_{\varepsilon }+R_{\varepsilon
^{-1}}\varepsilon WR_{\varepsilon }\\&=\varepsilon R_{\varepsilon
^{-1}}R_{\varepsilon }L_{0}+\varepsilon ^{-1}\varepsilon
R_{\varepsilon ^{-1}}R_{\varepsilon }W=\varepsilon L_{0}+W.
\end{align*}
By Proposition~\ref{opers}, $\left( \varepsilon L_{0}\right) ^{\left( n\right) }$ is a
generator of a contraction semigroup in $X_{n}$ for any $n\geq 1$. Hence, the statement is
a direct consequence of Theorem~\ref{th1}. Note that the solution of \eqref{Cauchy-eps}
can be found recursively:
\begin{align}\nonumber
G_{t,\varepsilon }^{\left( n\right) }\bigl(x^\n \bigr) &=\left(
e^{\varepsilon tL_{0}^{\left( n\right) }}G_{0,\varepsilon }^{\left(
n\right) }\right) \bigl(x^\n \bigr) \\&\qquad+\int_{0}^{t}\left(
e^{\varepsilon \left( t-s\right) L_{0}^{\left( n\right) }}W^{\left(
n\right) }G_{s,\varepsilon }^{\left( n-1\right) }\right) \bigl(x^\n
\bigr) ds,~n\geq 1,\label{recurs-eps}
\end{align}
and  $G_{t,\varepsilon }^{\left( 0\right) } =G_{0,\varepsilon }^{\left( 0\right)}$.
\end{proof}

By \eqref{renexpr},
$\widehat{L}_{\eps,\mathrm{ren}}G(\eta)\rightarrow WG(\eta)$ as
$\eps\rightarrow0$ ($\eta\in\Ga_0$). Let \eqref{c-3}--\eqref{c-4}
hold. By Theorem~\ref{th1},
 for any $C>0$, the initial value problem
\begin{equation}\label{Cauchy-V}
\begin{aligned}
\dfrac{\partial }{\partial t}G_{t,V}\left( \eta \right) &=\left(
WG_{t,V}\right) \left( \eta \right), \\[2mm]
\left. G_{t,V}\right\vert _{t=0}&=G_{0,V}\in\L_C
\end{aligned}
\end{equation}
has a unique solution $G_{t,V}\in\L _{\rho(t,C)}\,$, with $\rho\left(
t,C\right) $ given by \eqref{weight}. This solution can be
constructed recursively, namely, $G_{t,V}^{\left( 0\right) }
=G_{0,V}^{\left( 0\right) }$ and
\begin{equation}\label{recurs-V}
G_{t,V}^{\left( n\right) }\bigl(x^\n \bigr) =G_{0,V}^{\left( n\right) }
\bigl(x^\n \bigr) +\int_{0}^{t}\left( W^{\left( n\right)
}G_{s,V}^{\left( n-1\right) }\right) \bigl(x^\n \bigr) ds,~n\geq 1.
\end{equation}

\begin{theorem}\label{th-conv-qo}
Suppose that  conditions \eqref{c-1}--\eqref{c-4}
and \eqref{dominate} hold. Let $C>0$ and let $\left\{
G_{0,V},G_{0,\varepsilon },\varepsilon >0\right\} \subset \L_C$ be such that
\begin{equation}
\bigl\Vert G_{0,\varepsilon }-G_{0,V}\bigr\Vert _{\mathcal{L}
_{C}}\rightarrow 0\quad \text{as }\varepsilon \rightarrow 0.
\label{init_conv}
\end{equation}
Then, for any $T>0$ and  any $r<\rho(T,C)$,
\begin{equation}
\sup_{t\in[0,T]}\bigl\Vert G_{t,\varepsilon }-G_{t,V}\bigr\Vert _{\L
_r}\rightarrow 0\quad \text{as }\varepsilon \rightarrow 0. \label{time_conv}
\end{equation}
\end{theorem}

\begin{proof}
By \eqref{recurs-eps} and  \eqref{recurs-V}, we have
\begin{align}
&\bigl\Vert G_{t,\varepsilon }^{\left( n\right) }-G_{t,V}^{\left( n\right) }
\bigr\Vert _{X_{n}}\nonumber \\
\leq&\, \bigl\Vert e^{\varepsilon tL_{0}^{\left( n\right) }}G_{0,\varepsilon
}^{\left( n\right) }-G_{0,V}^{\left( n\right) }\bigr\Vert _{X_{n}}\nonumber \\
&+\int_{0}^{t}\bigl\Vert e^{\varepsilon \left( t-s\right) L_{0}^{\left(
n\right) }}W^{\left( n\right) }G_{s,\varepsilon }^{\left( n-1\right)
}-W^{\left( n\right) }G_{s,V}^{\left( n-1\right) }\bigr\Vert _{X_{n}}ds\nonumber \\
\leq& \,\Bigl\Vert e^{\varepsilon tL_{0}^{\left( n\right) }}\bigl(
G_{0,\varepsilon }^{\left( n\right) }-G_{0,V}^{\left( n\right) }\bigr)
\Bigr\Vert _{X_{n}}+\Bigl\Vert \bigl( e^{\varepsilon tL_{0}^{\left( n\right)
}}-1\bigr) G_{0,V}^{\left( n\right) }\Bigr\Vert _{X_{n}}\nonumber \\
&+\int_{0}^{t}\Bigl\Vert e^{\varepsilon \left( t-s\right) L_{0}^{\left(
n\right) }}W^{\left( n\right) }\bigl( G_{s,\varepsilon }^{\left( n-1\right)
}-G_{s,V}^{\left( n-1\right) }\bigr) \Bigr\Vert _{X_{n}}ds\nonumber \\
&+\int_{0}^{t}\Bigl\Vert \bigl( e^{\varepsilon \left( t-s\right)
L_{0}^{\left( n\right) }}-1\bigr) W^{\left( n\right) }G_{s,V}^{\left(
n-1\right) }\Bigr\Vert _{X_{n}}ds\nonumber \\\intertext{By Proposition~\ref{opers},
$L_0^\n$ is the generator of a contraction
semigroup in $X_n$, hence we continue}
\leq& \,\bigl\Vert G_{0,\varepsilon }^{\left( n\right) }-G_{0,V}^{\left(
n\right) }\bigr\Vert _{X_{n}}+\Bigl\Vert \bigl( e^{\varepsilon
tL_{0}^{\left( n\right) }}-1\bigr) G_{0,V}^{\left( n\right) }\Bigr\Vert
_{X_{n}} \nonumber\\
&+\int_{0}^{t}\Bigl\Vert W^{\left( n\right) }\bigl( G_{s,\varepsilon
}^{\left( n-1\right) }-G_{s,V}^{\left( n-1\right) }\bigr) \Bigr\Vert
_{X_{n}}ds\nonumber \\
&+\int_{0}^{t}\Bigl\Vert \bigl( e^{\varepsilon \left( t-s\right)
L_{0}^{\left( n\right) }}-1\bigr) W^{\left( n\right) }G_{s,V}^{\left(
n-1\right) }\Bigr\Vert _{X_{n}}ds.\label{expr0}
\end{align}
By \eqref{normLc}, for any $n\in\N_0$
\begin{equation}\label{eq}
\frac{C^{n}}{n!}\bigl\Vert G_{0,\varepsilon }^{\left( n\right)
}-G_{0,V}^{\left( n\right) }\bigr\Vert _{X_{n}}\leq \bigl\Vert
G_{0,\varepsilon }-G_{0,V}\bigr\Vert _{\mathcal{L}_{C}}.
\end{equation}
Hence, condition (\ref{init_conv}) implies that $\bigl\Vert
G_{0,\varepsilon }^{\left( n\right) }-G_{0,V}^{\left( n\right) }\bigr\Vert
_{X_{n}}\rightarrow 0$ as $\varepsilon \rightarrow 0$. Since
$L_0^\n$ is a bounded generator of a strongly continuous contraction
semigroup on $X_n$ and since  $G_{0,V}^{\left( n\right) }\in X_{n}$, we get from \eqref{represent}:
\begin{equation*}
\sup_{t\in[0,T]}\Bigl\Vert \bigl( e^{\varepsilon tL_{0}^{\left( n\right) }}-1\bigr)
G_{0,V}^{\left( n\right) }\Bigr\Vert _{X_{n}}\leq \eps T\bigl\Vert L_{0}^{\left( n\right) }G_{0,V}^{\left( n\right) }\bigr\Vert _{X_{n}} \rightarrow 0\quad \text{as }
\varepsilon \rightarrow 0.
\end{equation*}
Next, by the inclusion $W^{\left( n\right)
}G_{s,V}^{\left( n-1\right) }\in X_{n}$, formula
\eqref{represent} also yields that, for any fixed $t>0$
and any $s\in \left[ 0,t\right]$,
\begin{align*}
&\Bigl\Vert \bigl( e^{\varepsilon \left( t-s\right) L_{0}^{\left( n\right)
}}-1\bigr) W^{\left( n\right) }G_{s,V}^{\left( n-1\right) }\Bigr\Vert
_{X_{n}}\leq\int_0^{\eps (t-s)}\bigl\Vert L_{0}^{\left( n\right) }W^{\left( n\right) }G_{\tau,V}^{\left( n\right) }\bigr\Vert _{X_{n}}d\tau\\
\intertext{By Proposition~\ref{opers}, we continue with $A:=\dfrac{c_1+c_2}{2}$ and
 $B:=c_3+c_4$:}
 \leq  &\, ABn^2\left( n-1\right)^2 \int_0^{\eps (t-s)}\bigl\Vert G_{\tau,V}^{\left(
n-1\right) }\bigr\Vert _{X_{n-1}} d\tau\\\intertext{By \eqref{normLc},
similarly to \eqref{eq}, we estimate:}\leq &\,ABn^2\left( n-1\right)^2  \int_0^{\eps (t-s)}({\rho(\tau,C)})^{-(n-1)}\left( n-1\right) !\|G_{\tau,V}\|_{\L_{\rho(\tau,C)}}d\tau\\\intertext{For $0\leq
s\leq t\leq T$ and  $0<\eps<1$, by \eqref{weight} and  \eqref{maincompar}, we continue:}&\leq \eps TABn^2\left( n-1\right)^2 ({\rho(T,C)})^{-(n-1)} \left( n-1\right) !\|G_{0,V}\|_{\L_{C}}.
\end{align*}
Therefore,
\begin{equation*}
\sup_{t\in[0,T]}\int_{0}^{t}\bigl\Vert \bigl( e^{\varepsilon \left( t-s\right) L_{0}^{\left(
n\right) }}-1\bigr) W^{\left( n\right) }G_{s,V}^{\left( n-1\right) }
\bigr\Vert _{X_{n}}ds\rightarrow 0\quad \text{as }  \varepsilon \rightarrow 0.
\end{equation*}
Suppose now that
\begin{equation}\label{assumpt}
\sup_{t\in[0,T]}\bigl\Vert G_{t,\varepsilon }^{\left( n-1\right) }-G_{t,V}^{\left(
n-1\right) }\bigr\Vert _{X_{n-1}}\rightarrow 0\quad \text{as }  \varepsilon \rightarrow 0.
\end{equation}
Then, by \eqref{estWn},
\begin{align*}
&\sup_{t\in[0,T]}\int_{0}^{t}\Bigl\Vert W^{\left( n\right) }\bigl( G_{s,\varepsilon }^{\left(
n-1\right) }-G_{s,V}^{\left( n-1\right) }\bigr) \Bigr\Vert
_{X_{n}}ds\\\leq &\,Bn(n-1)\,T\sup_{t\in[0,T]}\bigl\Vert G_{t,\varepsilon }^{\left( n-1\right) }-G_{t,V}^{\left(
n-1\right) }\bigr\Vert _{X_{n-1}}\rightarrow 0\quad \text{as }  \varepsilon \rightarrow 0.
\end{align*}
Note that, for any $t\in[0,T]$, using  \eqref{eq}, we obtain
$$\bigl\vert G_{t,\varepsilon }^{\left( 0\right) }-G_{t,V}^{\left( 0\right) }
\bigr\vert=\bigl\vert G_{0,\varepsilon }^{\left( 0\right) }-G_{0,V}^{\left( 0\right) }
\bigr\vert\rightarrow 0\quad\text{as }\varepsilon \rightarrow 0.$$  As a result, by the induction principle,
we conclude from \eqref{expr0}   that,
for any $n\in\N_0$,
\begin{equation}\label{ind}
\sup_{t\in[0,T]}\bigl\Vert G_{t,\varepsilon }^{\left( n\right) }-G_{t,V}^{\left( n\right) }
\bigr\Vert _{X_{n}}\rightarrow 0\quad \text{as } \varepsilon \rightarrow 0.
\end{equation}

Let now $0<r<\rho(T,C)$.
Then
\begin{equation}\label{ser}
\sup_{t\in[0,T]}\bigl\Vert G_{t,\varepsilon }-G_{t,V}\bigr\Vert _{\mathcal{L}_r} \leq\sum_{n=0}^{\infty }\frac{r^n}{n!}\sup_{t\in[0,T]}\bigl\Vert G_{t,\varepsilon }^{\left( n\right)
}-G_{t,V}^{\left( n\right) }\bigr\Vert _{X_{n}}.
\end{equation}
Hence, by \eqref{ind} and \eqref{ser}, to prove the theorem it suffices
 to show that the series \eqref{ser} converges uniformly
in $\eps$.
But the latter series may be estimated, similarly to the considerations
above:
\begin{align*}
&\sum_{n=0}^{\infty }\frac{r^n}{n!}\sup_{t\in[0,T]}\Bigl(\bigl\Vert
G_{t,\varepsilon }^{\left( n\right) }\bigr\Vert_{X_n}+\bigl\Vert
G_{t,V}^{\left( n\right) }\bigr\Vert _{X_{n}}\Bigr)\\ \leq &\,
\sum_{n=0}^{\infty
}\frac{r^n}{n!}\sup_{t\in[0,T]}\Bigl((\rho(t,C))^{-n}n!\bigl\Vert
G_{t,\varepsilon }\bigr\Vert
_{\L_{\rho(t,C)}}+(\rho(t,C))^{-n}n!\bigl\Vert G_{t,V}\bigr\Vert
_{\L_{\rho(t,C)}}\Bigr)\\ \leq &\,
\Bigl(\|G_{0,V}\|_{\L_{C}}+\sup_{\eps>0}\|G_{0,\eps}\|_{\L_{C}}\Bigr)\sum_{n=0}^{\infty
}\biggl(\frac{r^{} }{\rho(T,C)}\biggr)^n <\infty,
\end{align*}
since  $G_{0,\eps}\rightarrow G_{0,V}$ in
$\L_C$ yields $\sup_{\eps>0}\|G_{0,\eps}\|_{\L_{C}}<\infty$.
\end{proof}

\begin{proposition}\label{prop-conv-cf}
Let the conditions of Theorem~\ref{th-cr} hold. Let $C_0>0$,
$T=\frac{1}{BC_0}$,  and $\bigl\{ k_{0,V},
k_{0,\eps},\eps>0\bigr\}\subset \K_{C_0}$. Then, for any $t\in(0,T)$,
there exist functions $\bigl\{k_{t,V},k_{t,\eps},
\eps>0\bigr\}\subset \K_{C_t}$, with $C_t$ given by \eqref{Ct}, such
that,  for any $G\in\Bbs$,
\[
\frac{\partial}{\partial t}\langle\!\langle
G,k_{t,\eps}\rangle\!\rangle=\langle\!\langle
(\eps L_0+W)G,k_{t,\eps}\rangle\!\rangle,
\quad \frac{\partial}{\partial t}\langle\!\langle
G,k_{t,V}\rangle\!\rangle=\langle\!\langle
WG,k_{t,V}\rangle\!\rangle.
\]
Moreover,
$\|k_{t,\eps}\|_{\K_{C_t}}\leq\|k_{0,\eps}\|_{\K_{C_0}}$,
$\|k_{t,V}\|_{\K_{C_t}}\leq\|k_{0,V}\|_{\K_{C_0}}$.
If, additionally,
\begin{equation}\label{initconv}
\lim_{\eps\rightarrow0}\|k_{0,\eps}-k_{0,V}\|_{\K_{C_0}}=0,
\end{equation}
then for any $T'\in(0,T)$,  $r_0>C_{T'}$, and $G_0\in\L_{r_0}$,
\begin{equation}
\sup_{t\in[0,T']}\bigl| \langle\!\langle G_0,k_{t,\eps}-k_{t,V}\rangle\!\rangle\bigr|\rightarrow
0\quad \text{as } \eps\rightarrow 0.
\end{equation}
\end{proposition}
\begin{proof}
The first part of the statement follows from  Theorem~\ref{th-cr}.
Since the function $[0,T']\ni t\mapsto C_t$ is (strictly) increasing,
we have $\bigl\{k_{t,V},k_{t,\eps}, \eps>0\bigr\}\subset
\K_{C_{T'}}\subset \K_{r_0}$. Moreover, by the proof of
Theorem~\ref{th-cr}, for any $G_0\in\L_{r_{0}}\subset \L_{C_t}$,
\begin{equation}\label{pairings}
\langle\!\langle G_0,k_{t,\eps}\rangle\!\rangle=\langle\!\langle G_{t,\eps},k_{0,\eps}\rangle\!\rangle, \quad \langle\!\langle G_0,k_{t,V}\rangle\!\rangle=\langle\!\langle G_{t,V},k_{0,V}\rangle\!\rangle,
\end{equation}
where $G_{t,\eps}$ and $G_{t,V}$ are solutions to \eqref{Cauchy-eps}
and \eqref{Cauchy-V}, respectively. Therefore, by Theorem~\ref{th1}, $\bigl\{G_{t,\eps}, G_{t,V}\bigr\}\subset
\L_{\rho(t,r_0)}$, where the function $\rho$
is given by \eqref{weight}. Since $\rho$ is (strictly) increasing in the second
variable and since $r_0>C_{T'}$, we have
$$\rho(T',r_0)>\rho(T',C_{T'})=C_0.$$
Applying Theorem~\ref{th-conv-qo} with $C=r_0$, $r=C_0$, $T=T'$, and $G_{0,\varepsilon }=G_{0,V}=G_{0}$
we obtain \begin{equation}
\sup_{t\in[0,T']}\bigl\Vert G_{t,\varepsilon }-G_{t,V}\bigr\Vert _{\L
_{C_0}}\rightarrow 0\quad \text{as } \varepsilon \rightarrow 0. \label{time_conv2}
\end{equation}
Then, by \eqref{pairings},
\begin{align}
&\quad\ \sup_{t\in[0,T']}\bigl|\langle\!\langle G_0,k_{t,\eps}-k_{t,V}\rangle\!\rangle \bigr|=\sup_{t\in[0,T']}\bigl|\langle\!\langle G_{t,\eps},k_{0,\eps}\rangle\!\rangle-\langle\!\langle G_{t,V},k_{0,V}\rangle\!\rangle \bigr|\nonumber\\&\leq\sup_{t\in[0,T']}\bigl|\langle\!\langle G_{t,\eps}-G_{t,V},k_{0,\eps}\rangle\!\rangle \bigr|+\sup_{t\in[0,T']}\bigl|\langle\!\langle G_{t,V},k_{0,\eps}-k_{0,V}\rangle\!\rangle \bigr|\nonumber\\&\leq\sup_{t\in[0,T']} \|G_{t,\eps}-G_{t,V}\|_{\L_{C_0}}\|k_{0,\eps}\|_{\K_{C_0}}\nonumber\\&\quad+\sup_{t\in[0,T']}\|G_{t,V}\|_{
\L_{C_0}} \|k_{0,\eps}-k_{0,V}\|_{\K_{C_0}}.\label{est5}
\end{align}
Since $k_{0,\eps}\rightarrow k_{0,V}$ in
$\K_{C_0}$, we get $\sup_{\eps>0}\|k_{0,\eps}\|_{\K_{C_0}}<\infty$. Then,
by \eqref{time_conv2}, the first summand in \eqref{est5} converges
to $0$ as $\eps\rightarrow0$. Next, by \eqref{maincompar},
\[
\|G_{t,V}\|_{
\L_{C_0}}=\|G_{t,V}\|_{
\L_{\rho(t,C_t)}} \leq \|G_{0}\|_{\L_{C_t}}\leq \|G_{0}\|_{\L_{C_{T'}}}\leq\|G_0\|_{\L_{r_0}}.
\]
Hence, by \eqref{initconv}, the second summand in \eqref{est5} also
converges to $0$ as $\eps\rightarrow0$, which proves  the second part of the statement.
\end{proof}

We will  now show that the evolution $k_{0,V}\mapsto
k_{t,V}$ satisfies our second requirement on the initial
scaling $L\mapsto L_\eps$, namely, $e_\la(p_0,\eta)\mapsto
e_\la(p_t,\eta)$.
\begin{proposition}
Let the conditions of Theorem~\ref{th-cr} hold. Then for any $G\in \Bbs$ and  any $k\in\K_C$ with $C>0$,
\begin{equation}\label{dualoper}
\int_{\Gamma _{0}}\left( WG\right) \left( \eta \right) k\left(
\eta
\right) d\lambda \left( \eta \right)=\int_{\Gamma _{0}}G\left( \eta \right) (W^*k)\left(
\eta
\right) d\lambda \left( \eta \right),
\end{equation}
where
\begin{align}
(W^*k)(\eta)=&\sum_{y_{1}\in \eta }\int_{\mathbb{R}^{d}}\int_{\mathbb{R}
^{d}}\tilde{c}\left(  x_{1},x_{2} ,
y_{1}\right)  k\left( \eta \cup x_{2}\cup x_{1}\setminus y_{1}\right)
dx_{2}dx_{1}\nonumber \\
& -\sum_{x_{1}\in \eta }\int_{\mathbb{R}^{d}}a_1(x_1,x_2)k\left( \eta \cup x_{2}\right) dx_{2}. \label{W-dual}
\end{align}
Here the functions $\tilde{c}$ and $a_1$ are defined by \eqref{c-tilde} and \eqref{a1},
respectively.
\end{proposition}
\begin{proof}
First, we  note that, under conditions \eqref{c-1}--\eqref{c-4}, for $G\in \Bbs$ and  $k\in\K_C$, both integrals
in \eqref{dualoper} are well defined.  Then, using \eqref{W} and e.g.\
\cite[Lemma 1]{FKK2010a}, we have
\begin{align*}
& \int_{\Gamma _{0}}\left( WG\right) \left( \eta \right) k\left(
\eta
\right) d\lambda \left( \eta \right) \\
=&\int_{\Gamma _{0}}\int_{\mathbb{R}^{d}}\sum_{x_{1}\in \eta }\int_{\mathbb{
R}^{d}}\tilde{c}\left(  x_{1},x_{2} ,
y_{1}\right)  G\left( \eta \setminus
x_{1}\cup y_{1}\right) dy_{1}k\left( \eta \cup x_{2}\right) dx_{2}d\lambda
\left( \eta \right) \\
& -\int_{\Gamma _{0}}\int_{\mathbb{R}^{d}}\sum_{x_{1}\in \eta }\int_{\mathbb{
R}^{d}}\tilde{c}\left(  x_{1},x_{2} ,
y_{1}\right)   G\left( \eta \right)
dy_{1}k\left( \eta \cup x_{2}\right) dx_{2}d\lambda \left( \eta \right) \\
=&\int_{\Gamma _{0}}\int_{\mathbb{R}^{d}}\int_{\mathbb{R}
^{d}}\sum_{y_{1}\in \eta }\tilde{c}\left(  x_{1},x_{2} ,
y_{1}\right)  G\left( \eta \right) k\left( \eta \cup x_{2}\cup x_{1}\setminus y_{1}\right)
dx_{2}dx_{1}d\lambda \left( \eta \right) \\
& -\int_{\Gamma _{0}}\int_{\mathbb{R}^{d}}\sum_{x_{1}\in \eta }\int_{\mathbb{
R}^{d}}\tilde{c}\left(  x_{1},x_{2} ,
y_{1}\right) G\left( \eta \right)
dy_{1}k\left( \eta \cup x_{2}\right) dx_{2}d\lambda \left( \eta \right) ,
\end{align*}
which proves the statement.
\end{proof}

Thus, for any $p_t\in L^\infty(\X)$,
\begin{align*}
(W^*e_\la(p_t))(\eta)=&\sum_{y\in \eta }e_\la(p_t,\eta\setminus y)\int_{\mathbb{R}^{d}}\int_{\mathbb{R}
^{d}}\tilde{c}\left(  x_{1},x_{2} ,
y\right)  p_t(x_1)p_t(x_2)dx_{2}dx_{1} \\
& -\sum_{y\in \eta }e_\la(p_t,\eta\setminus y)p_t(y)\int_{\mathbb{R}^{d}}a_1(y,x_2)p_t(x_2) dx_{2}.
\end{align*}
On the other hand, if $\frac{d}{d t}p_t$ exists, then
\[
\frac{\partial}{\partial t}e_\la(p_t,\eta)=\sum_{y\in \eta }e_\la(p_t,\eta\setminus y)\frac{\partial}{\partial t}p_t(y).
\]
Therefore, there exists a (point-wise) solution
$k_t=e_\la(p_t,\eta)$
of the initial value problem
\begin{equation}
\frac{\partial k_t}{\partial t}=W^*k_t,\qquad k_t\bigm|_{t=0}=e_\la(p_0,\eta),
\end{equation}
provided $p_t$ satisfies the non-linear Vlasov-type
equation
\begin{align}
\frac{\partial }{\partial t}p _{t}\left( x\right) =&\int_{\mathbb{R}
^{d}}\int_{\mathbb{R}^{d}}\tilde{c}(y_1,y_2,x) p _{t}\left( y_{1}\right) p
_{t}\left( y_{2}\right) dy_{1}dy_{2}\nonumber \\
& -p _{t}\left( x\right) \int_{\mathbb{R}^{d}}a_1(x,x_2) p _{t}\left( x_{2}\right) dx_{2}.
\label{Vlasoveqn}\end{align}
If the  symmetry condition \eqref{utdey7eu} holds, then we may rewrite \eqref{Vlasoveqn} in
the Boltzmann-type form
\begin{align}
\frac{\partial }{\partial t}p _{t}\left( x\right) =&\int_{\mathbb{R}
^{d}}\int_{\mathbb{R}^{d}}\int_{\mathbb{R}^{d}}c\left(
x,x_{2},y_{1},y_{2}\right)\nonumber\\&\qquad\times  \left[ p _{t}\left( y_{1}\right) p
_{t}\left( y_{2}\right) -p _{t}\left( x\right) p _{t}\left(
x_{2}\right) \right] dy_{1}dy_{2}dx_{2}.
\label{Boltzeqn}\end{align}

We are interested in positive bounded solutions of
\eqref{Vlasoveqn}.

\begin{proposition}\label{sol}
Let $C>0$ and let  $0\leq p _0\in L^\infty(\X)$
with $\|p_0\|_{L^\infty(\X)}\leq C$. Assume
that \eqref{c-1} and \eqref{c-4} hold and, moreover,
\begin{equation}\label{chr}
\int_\X\int_\X c(y,u_1,x,u_2)du_1du_2\leq\int_\X\int_\X c(x,y,u_1,u_2)du_1du_2.
\end{equation}
Then, for any $T>0$, there exists a function $0\leq p_t\in
L^\infty(\X)$, $t\in[0,T]$, which solves \eqref{Vlasoveqn} and,
moreover,
\begin{equation}\label{uniest}
\max_{t\in[0,T]}\|p_t\|_{L^\infty(\X)}\leq C.
\end{equation}
This functions is a unique non-negative solution to
\eqref{Vlasoveqn} which satisfies \eqref{uniest}.
\end{proposition}

\begin{proof}
Let us fix an arbitrary $T>0$ and define the Banach space
$X_T:=C([0,T],L^\infty(\X))$ of all continuous functions on $[0,T]$
with values in $L^\infty(\X)$; the norm on $X_T$ is given by
$$
\|u\|_T:=\max\limits_{t\in[0,T]}\|u_t\|_{L^\infty(\X)}.
$$
We denote by $X_T^+$ the cone of all nonnegative functions from
$X_T$. For a given $C>0$, denote  by $B_{T,C}^+$ the set of all
functions $u$ from $X_T^+$ with $\|u\|_T\leq C$.

Let $\Phi$ be a  mapping which
assigns to any $v\in X_T$ the solution
$u_t$
of the linear Cauchy problem
\begin{equation}\label{le}
\begin{aligned}
\frac{\partial }{\partial t}u_{t}\left( x\right) =&-u _{t}\left( x\right) \int_{\mathbb{R}^{d}}a_1(x,y) v_{t}\left( y\right) dy\\
&+\int_{\mathbb{R}
^{d}}\int_{\mathbb{R}^{d}}\tilde{c}(y_1,y_2,x) v_{t}\left( y_{1}\right) v_{t}\left( y_{2}\right) dy_{1}dy_{2},\\
u_t(x)\Bigr|_{t=0}=&\,p_0(x),
\end{aligned}
\end{equation}
namely,
\begin{align}\nonumber
(\Phi v)_t(x)&=\exp\biggl\{ -\int_0^t \int_{\mathbb{R}^{d}}a_1(x,y) v_{s}\left( y\right) dyds\biggr\}p_0(x)\\&\quad+\int_0^t
\exp\biggl\{ -\int_s^t \int_{\mathbb{R}^{d}}a_1(x,y) v_{\tau}\left( y\right) dyd\tau\biggr\}\nonumber\\&\qquad\times\int_{\mathbb{R}
^{d}}\int_{\mathbb{R}^{d}}\tilde{c}(y_1,y_2,x) v_{s}\left( y_{1}\right) v_{s}\left( y_{2}\right) dy_{1}dy_{2}ds.\label{Phiexpr}
\end{align}
Clearly, $v_t\geq0$ implies $(\Phi v)_t\geq0$.
Moreover, $v\in X_T^{+}$ yields
\[
\bigl|(\Phi v_t)(x)\bigr|\leq |p_0(x)|+c_4T\|v\|_T^2.
\]
Therefore, $\Phi:X_T^{+}\rightarrow X_T^{+}$. Obviously,
$u_t$ solves \eqref{Vlasoveqn} if and only
if $u$ is a fixed point of the map $\Phi$.

Since $0\leq p_0(x)\leq C$ for a.a. $x\in\X$,
 we get from
\eqref{Phiexpr} and \eqref{chr},  for any $v\in B_{T,C}^+$,
\begin{align*}
0\leq(\Phi v)_t(x)&\leq C\exp\biggl\{ -\int_0^t \int_{\mathbb{R}^{d}}a_1(x,y) v_{s}\left( y\right) dyds\biggr\}\\&\quad+C\int_0^t
\exp\biggl\{ -\int_s^t \int_{\mathbb{R}^{d}}a_1(x,y) v_{\tau}\left( y\right) dyd\tau\biggr\}\int_{\mathbb{R}
^{d}}a_1(x,y) v_{s}\left( y\right) dyds\\&\leq C\exp\biggl\{ -\int_0^t \int_{\mathbb{R}^{d}}a_1(x,y) v_{s}\left( y\right) dyds\biggr\}\\&\quad
+C\int_0^t \frac{\partial}{\partial s}
\exp\biggl\{ -\int_s^t \int_{\mathbb{R}^{d}}a_1(x,y) v_{\tau}\left( y\right) dyd\tau\biggr\}ds=C.\nonumber
\end{align*}
 Therefore, $\Phi: B_{T,C}^+\rightarrow B_{T,C}^+$.

Let us choose any $\Upsilon\in(0,T)$ such that
$2C(c_3+c_4)\Upsilon<1$. Clearly, $v\in B^+_{T,C}$ implies $v\in
B^+_{\Upsilon,C}$. By the proved above,  $\Phi:
B_{\Upsilon,C}^+\rightarrow B_{\Upsilon,C}^+$. Note  the elementary
inequalities  $|e^{-a}-e^{-b}|\leq |a-b|$ and
\[
|pe^{-a}-qe^{-b}|\leq e^{-a}|p-q|+qe^{-b}|e^{-(a-b)}-1|\leq
e^{-a}|p-q|+qe^{-b}|a-b|,
\]
which hold for any $a,b,p,q\geq0$. Hence, for any $v,w\in
B_{\Upsilon,C}^+$, $t\in[0,\Upsilon]$, we get
\begin{align*}
&\quad\bigl| (\Phi v)_t(x)-(\Phi w)_t(x)\bigr|\\&\leq\biggl| \int_0^t \int_{\mathbb{R}^{d}}a_1(x,y) v_{s}\left( y\right) dyds-\int_0^t \int_{\mathbb{R}^{d}}a_1(x,y)w_{s}\left( y\right) dyds\biggr| p_0(x)\\&\quad+\int_0^t
\exp\biggl\{ -\int_s^t \int_{\mathbb{R}^{d}}a_1(x,y) v_{\tau}\left( y\right) dyd\tau\biggr\}\nonumber\\&\qquad\times\biggl|\int_{\mathbb{R}
^{d}}\int_{\mathbb{R}^{d}}\tilde{c}(y_1,y_2,x) v_{s}\left( y_{1}\right) v_{s}\left( y_{2}\right) dy_{1}dy_{2}\\&\qquad\qquad-\int_{\mathbb{R}
^{d}}\int_{\mathbb{R}^{d}}\tilde{c}(y_1,y_2,x)w_{s}\left( y_{1}\right) w_{s}\left( y_{2}\right) dy_{1}dy_{2}\biggr|ds\\&\quad+\int_0^t
\exp\biggl\{ -\int_s^t \int_{\mathbb{R}^{d}}a_1(x,y)w_{\tau}\left( y\right) dyd\tau\biggr\}\nonumber\\&\qquad\times\int_{\mathbb{R}
^{d}}\int_{\mathbb{R}^{d}}\tilde{c}(y_1,y_2,x)w_{s}\left( y_{1}\right) w_{s}\left( y_{2}\right) dy_{1}dy_{2}\\&\qquad\times
\biggl|\int_s^t \int_{\mathbb{R}^{d}}a_1(x,y) v_{\tau}\left( y\right) dyd\tau - \int_s^t \int_{\mathbb{R}^{d}}a_1(x,y) w_{\tau}\left( y\right) dyd\tau\biggr|ds\\&=:I_1+I_2+I_3.
\end{align*}
By \eqref{c-1}, \eqref{c-4}, and \eqref{chr},
we have
\begin{align*}
I_1&\leq Cc_3\|v-w\|_\Upsilon\Upsilon,\\
I_3&\leq Cc_3\|v-w\|_\Upsilon t\int_0^t \frac{\partial}{\partial s}
\exp\biggl\{ -\int_s^t \int_{\mathbb{R}^{d}}a_1(x,y) w_{\tau}\left( y\right) dyd\tau\biggr\}ds\\&\leq
Cc_3\|v-w\|_\Upsilon \Upsilon,
\\I_2&\leq2Cc_4\|v-w\|_\Upsilon\Upsilon.
\end{align*}
Therefore, $$\|\Phi v-\Phi w\|_\Upsilon \leq 2C(c_3+c_4)\Upsilon
\|v-w\|_\Upsilon$$ for any $v,w\in B_{\Upsilon,C}^+$. Since
$B_{\Upsilon,C}^+$ is a metric space (with the metric induced by the norm
$\|\cdot\|_\Upsilon$) and since $2C(c_3+c_4)\Upsilon<1$, there exists a
unique $u^*\in B_{\Upsilon,C}^+$ such that $u^*=\Phi(u^*)$. Hence,
$u^*_t$ solves \eqref{Vlasoveqn} for $t\in[0,\Upsilon]$. Since
$0\leq u^*_\Upsilon(x)\leq C$ for a.a.\ $x\in\X$, we can consider
the equation \eqref{Vlasoveqn} with the initial value
$p_t(x)\bigr|_{t=\Upsilon} =u^*_\Upsilon(x)$. Then we obtain a
unique non-negative solution which satisfies
$\max_{t\in[\Upsilon,2\Upsilon]}\|u_t\|_{L^\infty(\X)}\leq C$, and so
on. As a result, we obtain a solution of \eqref{Vlasoveqn} on $[0,T]$.
The uniqueness among all solutions from $B_{T,C}^+$ is now obvious.
\end{proof}

\begin{remark}
Note that, in the proof of
Proposition~\ref{sol}, we did not  use the property \eqref{lhgu6r75}.  On the other hand,
all considerations remain true if, instead  of
\eqref{chr}, we assume that
\begin{equation}\label{chrsym}
\int_\X\int_\X c(u_1,y,x,u_2)du_1du_2\leq\int_\X\int_\X c(x,y,u_1,u_2)du_1du_2.
\end{equation}
Suppose we have an expansion
\begin{equation}\label{sum}
c(\{x_1,x_2\},\{y_1,y_2\})=c'(x_1,x_2,y_1,y_2)+c''(x_1,x_2,y_1,y_2)
\end{equation}
where $c',c''$ are functions which satisfy
conditions \eqref{chr} and \eqref{chrsym}, respectively, as well
as conditions \eqref{c-1} and  \eqref{c-4} (but do not necessarily satisfy \eqref{lhgu6r75}). It is easy to see that all considerations in
the proof of Proposition~\ref{sol} remain true for this $c$. Let us give a quite natural example of functions $c',\, c''$.
\end{remark}

\begin{example}[cf. \cite{FKKL2011a}]
Let $c$ be given by \eqref{sum} with
\[
c'(x_1,x_2,y_1,y_2)=\varkappa
a(x_1-y_1)a(x_2-y_2)\bigl[b(x_1-x_2)+b(y_1-y_2)\bigr].
\] and
$$c''(x_1,x_2,y_1,y_2)=c'(x_2,x_1,y_1,y_2).$$
Here $\varkappa>0$, $0\leq a,b\in L^1(\X)\cap L^\infty (\X)$,
$\|a\|_{L^1(\X)}=\|b\|_{L^1(\X)}=1$ and $b$ is an even function.
Then the condition \eqref{chrsym} for $c''$ coincides with the
condition \eqref{chr} for $c'$. Note also that
\eqref{c-1}--\eqref{c-4} are satisfied  for $c$. For example, $c_4=4\varkappa$
and
\[
c_1\leq2\varkappa
\|b\|_{L^\infty(\X)}+2\varkappa\|a\|_{L^\infty(\X)}
\|b\|_{L^\infty(\X)}<\infty.
\]
Next, let us check whether \eqref{chr} holds for $c'$. We have
\begin{align*}
&\quad\ \varkappa\int_\X\int_\X
c'(y,u_1,x,u_2)du_1du_2\\&=\varkappa\int_\X\int_\X
a(y-x)a(u_1-u_2)\bigl(b(y-u_1)+b(x-u_2)\bigr)du_1du_2= 2\varkappa
a(y-x)
\end{align*}
and
\begin{align*}
&\quad \varkappa \int_\X\int_\X
c'(x,y,u_1,u_2)du_1du_2\\&=\varkappa\int_\X\int_\X
a(x-u_1)a(y-u_2)\bigl(b(x-y)+b(u_1-u_2)\bigr)du_1du_2\\&= \varkappa
b(x-y) +\varkappa\int_\X\int_\X a(x-u_1)a(y-u_2)b(u_1-u_2)du_1du_2.
\end{align*}
Since $b$ is even,  \eqref{chr} holds if, for example, $2
a(x)\leq b(x)$, $x\in\X$.
\end{example}

\section*{Acknowledgements}

 The authors acknowledge the financial support of the SFB 701 ``Spectral
structures and topological methods in mathematics'', Bielefeld
University.

\end{document}